\documentclass[12pt]{amsart}
\usepackage[active]{srcltx}
\usepackage{color,graphicx,shortvrb}

\usepackage{amsmath,amssymb,amsthm}
\usepackage[colorlinks, bookmarks=true]{hyperref}

\theoremstyle{plain}
\newtheorem{theorem}{Theorem}[section]
\newtheorem{corollary}[theorem]{Corollary}
\newtheorem{lemma}[theorem]{Lemma}
\newtheorem{proposition}[theorem]{Proposition}

\theoremstyle{definition}

\newtheorem{remark}[theorem]{Remark}

\renewcommand{\leq}{\leqslant}
\renewcommand{\geq}{\geqslant}\usepackage{amssymb}

\newcommand{\vr}{\varepsilon}

\newcommand{\be}{\begin{equation}}
\newcommand{\ee}{\end{equation}}

\newcommand{\ii}{\mathcal{I}}
\newcommand{\tra}{\mathbf{tr}}
\newcommand{\diag}{\mathrm{diag}}
\newcommand{\bs}{\mathbf{S}}
\newcommand{\ks}{{\mathbf{K}}_\sigma}
\newcommand{\eks}{E_{\ks}}
\newcommand{\ssim}{\overset{*}{\sim}}
\newcommand{\nn}{\mathbf{n}}
\newcommand{\bfy}{\mathbf{Y}}

\newcommand{\om}{{\mathcal{O}}_m}
\newcommand{\subsp}{\mathbf{S}}

\newcommand{\fami}{\boldsymbol{\mathfrak{F}}}
\newcommand{\comp}{\boldsymbol{\mathfrak{C}}}
\newcommand{\xd}{X^d}
\newcommand{\Q}{\mathbb{Q}}
\newcommand{\coi}[1]{\overset{#1}{\simeq}}

\def \C{{\mathbb C}}

\def \N{{\mathbb N}}

\def \R{\mathbb R}

\def \hra{\hookrightarrow}

\def \dim{{\mathrm{dim}} \, }
\def \rank{{\mathrm{rank}} \, }
\def \ker{{\mathrm{ker}} \, }
\def \ball{{\mathrm{Ba}}}

\renewcommand{\span}{\mathrm{span}}
\def \vr{\varepsilon}

\def \ran{{\mathrm{ran}} \, }
\def \rank{{\mathrm{rank}} \, }

\def \diag{{\mathrm{diag}} \, }

\def \eqalign#1{\null\,\vcenter{\openup\jot 
   \ialign{\strut\hfil$\displaystyle{##}$&$
      \displaystyle{{}##}$\hfil \crcr#1\crcr}}\,}

\newcommand{\row}{\mathbf{R}}
\newcommand{\col}{\mathbf{C}}
\newcommand{\bsp}{\row \oplus \col}
\newcommand{\is}{{\mathcal{S}}}

\begin{document}



\numberwithin{equation}{section}

\pagestyle{headings}


\title[Subspace structure]{Subspace structure of some operator and Banach spaces}

\author{T.~Oikhberg and C.~Rosendal}

\address{
Department of Mathematics, The University of California at Irvine, Irvine CA 92697, {\it and}
Department of Mathematics, University of Illinois at Urbana-Champaign, Urbana, IL 61801}
\email{toikhber@math.uci.edu}

\address{
Department of Mathematics, Statistics, and Computer Science (M/C 249)
University of Illinois at Chicago
Chicago, IL 60607-7045}
\email{rosendal@math.uic.edu}





\begin{abstract}
We construct a family of separable Hilbertian operator spaces, such that the
relation of complete isomorphism between the subspaces of each member
of this family is complete $\ks$.
We also investigate some interesting properties of completely
unconditional bases of the spaces from this family.
In the Banach space setting, we construct a space for which the relation
of isometry of subspaces is equivalent to equality of real numbers.
\end{abstract}

\maketitle



\section{Introduction and main results}\label{intro}

Recently, there has been much progress in describing the complexity
of various relations between subspaces of a given separable Banach space.
The reader is referred to \cite{FG, FLRo, FRo} for the known results on
the relations of isomorphism, biembeddability, and more.
Isometry and local equivalence (finite representability) are handled
in \cite{Me} and \cite{GJS}, respectively.

In this paper, we consider an operator space analogue of this problem
(see Section~\ref{op_sp} for a brief introduction into operator spaces).
The Effros-Borel structure on the set
$\bs(Z)$ of infinite-dimensional subspaces of a separable operator space $Z$
is defined in the same way as for Banach spaces, see e.g.~Chapter 12 of \cite{Ke},
or \cite{FRo}. Reasoning as in Section 2 of \cite{FRo}, we show that the
relations 
of complete isomorphism, complete biembeddability, and such,
defined on $\bs(Z)$, are analytic equivalence relations.
How ``simple'' can these relations get, without being trivial?
The main result of this paper provides a partial answer to this question.

\begin{theorem}\label{intro:Ksigma}
There exists a family $\fami$ of operator spaces, such that, for any
operator space $X \in \fami$, the following is true:
\begin{enumerate}
\item
$X$ is isometric to $\ell_2$.
\item
The relation of complete isomorphism and complete biembeddability
on $\bs(X)$ are
Borel bireducible to the complete $\ks$ relation.
\end{enumerate}
The family $\fami$ contains a continuum of operator spaces,
not completely isomorphic to each other.
\end{theorem}

It is possible to prove that
the relation of complete isometry on $\bs(X)$ ($X \in \fami$)
is Borel bireducible to the equality on $\R$. The proof proceeds along the
same lines as Theorem~\ref{bireducible}, but is exceedingly technical,
and not very illuminating. We therefore omit it, and present
a related Theorem~\ref{intro:banach} instead.

Each space from $\fami$ has its {\it canonical basis} (defined in
Section~\ref{sequences}), which is $1$-completely unconditional.
It turns out these bases have interesting properties of their own.

\begin{theorem}\label{intro:unique_basis}
Any subspace $Y$ of an operator space $X$ from the family $\fami$
has $1$-completely unconditional canonical basis.
Any $C$-completely unconditional basis in such a $Y$
is $\phi(C)$-equivalent (up to a permutation) to the
canonical basis of $Y$, with $\phi(C)$ polynomial in $C$.
\end{theorem}

For certain spaces $X$, a stronger result holds.

\begin{theorem}\label{intro:subbasis}
For any $a > 1$ there exists an operator space $X$, 
belonging to the family $\fami$,
such that the canonical basis in any subspace of $X$ is
$a$-equivalent to a subsequence of the canonical basis of $X$.
Consequently, every $C$-completely unconditional basic sequence in
$X$ is $\phi(C)$-equivalent (up to a permutation) to a subsequence of the
canonical basis, with $\phi(C)$ polynomial in $C$.
\end{theorem}

As part of our motivation lies in the field of Banach spaces, we mention
a few classical results related to the three theorems above.

No ``commutative'' counterpart of Theorem~\ref{intro:Ksigma}
has yet been obtained:
for a Banach space $E$, very little is known about the ``upper''
estimates on the complexity of the isomorphism relation on $\bs(E)$.
It is possible that the relation of isomorphism on $\bs(E)$
is complete analytic whenever $E$ is a separable Banach space,
not isomorphic to $\ell_2$.



Theorem~\ref{intro:unique_basis} shows (among other things)
that every subspace of $X$ has an unconditional basis.
One can raise a related question:
suppose every infinite dimensional subspace of a separable
Banach space $E$ has an unconditional basis.
Must $E$ be isomorphic to $\ell_2$? The answer to this question
appears to be unknown. We mention a few partial results. First,
any subspace of $E$ has the Approximation Property,
therefore, by \cite[Theorem 1.g.6]{LT2}, $E$ has to have type $2 - \vr$ and
cotype $2 + \vr$ for any $\vr > 0$. Furthermore, by \cite[Theorem 10.13]{PiVOL}, 
$E$ has weak cotype $2$. By \cite{KT95}, $E$ is $\ell_2$-saturated.
Finally, if $E = \ell_2(F)$ for some Banach space $F$, then
$E$ is isomorphic to a Hilbert space \cite{KT02}.

Searching for a ``commutative'' analogues of Theorem~\ref{intro:subbasis},
we pose the following question: suppose a Banach space $E$ has
an unconditional basis $(e_i)$, such that every subspace of $E$ has
an unconditional basis, equivalent to a subsequence of $(e_i)$. Must
$E$ be isomorphic to $\ell_2$?



Finally, in the Banach space setting, we construct a space $E$
whose subspace structure is ``very simple.''

\begin{theorem}\label{intro:banach}
The following equivalence relations between infinite dimensional subspaces of
$\R \oplus_1 \ell_2$ (or, in the complex case, $\C \oplus_1 \ell_2$)
are Borel bireducible to $(\R,=)$: (i) isometry,
(ii) having Banach-Mazur distance $1$;
(iii) isometric bi-embeddability; (iv) almost isometric bi-embeddability.
\end{theorem}

Recall that the Banach-Mazur distance between $Y$ and $Z$ is defined as
$d(Y,Z) = \inf \{ \|u\| \|u^{-1}\| : u \in B(Y,Z) \}$.
$Y$ is called {\it almost isometrically embeddable} into $Z$
if for every $\vr > 0$ there exists $Y_\vr \hra Z$ s.t. $d(Y,Y_\vr) < 1+\vr$
($d(\cdot , \cdot)$ denotes the Banach-Mazur distance between spaces).

The rest of the paper is organized as follows: Section~\ref{op_sp}
provides a brief introduction into operator spaces and c.b.~maps.
In Section~\ref{construct} we construct, for each contraction
$A \in B(\ell_2)$, a subspace $X(A)$ of $\row \oplus \col$,
and study the properties of the spaces $X(A)$. Further investigation is
carried out in Section~\ref{sequences}, where we concentrate on
the case when $A$ is compact. Unconditional bases
in the spaces $X(A)$ are described in Section~\ref{bases}.
In Section~\ref{classification} we show that, for the ``right''
compact contractions $A$, the relations of complete isomorphism and
complete biembeddability on $\bs(X(A))$ are Borel bireducible
to the complete $\ks$ relation.
Having gathered all the preliminary results, we prove Theorems
\ref{intro:Ksigma}, \ref{intro:unique_basis}, and \ref{intro:subbasis} in
Section~\ref{main_proofs}.
Finally, in Section~\ref{banach} we
prove Theorem~\ref{intro:banach}.

\section{Introduction into operator spaces}\label{op_sp}

As our paper deals primarily with operator spaces, we are compelled to
recall some basic definitions and facts about the topic. An interested
reader is referred to \cite{ER, Pa, PiINTRO} for more information.
A ({\it concrete}) {\it operator space} $X$ is, for us, just a closed
subspace of $B(H)$ ($H$ is a Hilbert space). If $X$ and $Y$ are operator spaces, embedded into $B(H)$ and $B(K)$ respectively, we defined the {\it minimal tensor
product} of $X$ and $Y$ (denoted simply by $X \otimes Y$) as the
closure of the algebraic tensor product $X \odot Y$ in $B(H \otimes_2 K)$.
It is common to denote $M_n \otimes X$ by $M_n(X)$. Here, $M_n = B(\ell_2^n)$
is the space of $n \times n$ matrices. We view $M_n(X)$ as the space
of $X$-valued $n \times n$ matrices, with the norm $\| \cdot \|_n$.
It is easy to see that the sequence of {\it matricial norms} $\| \cdot \|_n$
satisfies two properties ({\it Ruan's axioms}): (i) for any $v \in M_n(X)$,
$\alpha \in M_{n,k}$, and $\beta \in M_{k,n}$,
$\|(\beta \otimes I_X) v (\alpha \otimes I_X)\|_k \leq
\|\beta\| \|x\|_n \|\alpha\|$, and (ii) for any $v \in M_n(X)$ and
$w \in M_k(X)\|$, $\|v \oplus w\|_{k+n} = \max\{\|v\|_n, \|w\|_k\}$.
It turns out that the converse is also true. Suppose $X$ is a
{\it matricially normed space} -- that is, it is a Banach space, for which
the spaces $M_n(X)$ of $n \times n$ $X$-valued matrices are equipped
with the norms $\| \cdot \|_n$, satisfying (i) and (ii) above. Then
the norms $\| \cdot \|_n$ arise from an isometric embedding of $X$ into
$B(H)$, for some Hilbert space $H$. The spaces $X$ as above are sometimes called
{\it abstract operator spaces}.

It is easy to see that a subspace of an operator space is, again, an
operator space (we use the notation $\hra$ to denote one operator space
being a subspace of another). Moreover, a quotient and the dual of
an operator space can be again equipped with an operator space structure.
Once again, the reader is referred to \cite{ER, Pa, PiINTRO} for details.

A map $u$ from an operator space $X$ to an operator space $Y$ is called
{\it completely bounded} ({\it c.b.} for short) if its {\it c.b.~norm}
$$
\|u\|_{cb} = \sup_n \|I_{M_n} \otimes u\|_{B(M_n(X), M_n(Y))} =
\|I_{B(\ell_2)} \otimes u\|_{B(B(\ell_2) \otimes X, B(\ell_2) \otimes Y)}
$$
is finite. The set of all c.b.~maps from $X$ to $Y$ is denoted by $CB(X,Y)$. Clearly, $\|u\|_{cb} \geq \|u\|$, and $CB(X,Y) \subset B(X,Y)$
(the inclusion may be strict). The operator spaces $X$ and $Y$ are called
{\it completely isomorphic} ({\it completely isometric})
if there exists $u \in CB(X,Y)$ such that $u^{-1}$ is c.b.
(resp. $\|u\|_{cb} = \|u^{-1}\|_{cb} = 1$).
We shall use the notation $\simeq$ 
for complete isomorphism. 
Moreover, we say that $X$ is {\it $c$-completely isomorphic to $Y$}
($X \coi{c} Y$) if there exists $u \in CB(X,Y)$ with
$\|u\|_{cb} \|u^{-1}\|_{cb} \leq c$.
Complete embeddability and biembeddability are defined in the obvious way.

Suppose the operator space s $X$ and $Y$ are embedded into $B(H)$ and $B(K)$,
respectively. We define the {\it direct sum} of operator spaces $X$ and $Y$ (denoted by $X \oplus_\infty Y$, or simply $X \oplus Y$) by viewing $X \oplus Y$
as embedded into $B(H \oplus_2 K)$. Note that any $u \in M_n(X \oplus Y)$
has a unique expansion as $v \oplus w$, with $v \in M_n(X)$ and $w \in M_n(Y)$.
Then $\|u\| = \max\{\|v\|, \|w\|\}$.

Throughout this paper, we work with the row and column spaces. Recall that
a Hilbert space $H$ can be equipped with {\it row} and {\it column}
operator space structure, denoted by $H_\row$ and $H_\col$, respectively.
The space $H_\row$ is defined as the linear space of operators
$\xi \otimes \xi_0$, where $\xi_0$ is a fixed unit vector, and $\xi$
runs over $H$. Here, for $\xi \in H$ and $\eta \in K$, $\xi \otimes \eta$
denotes the operator in $B(H,K)$, defined by
$(\xi \otimes \eta)\zeta = \langle \zeta, \xi \rangle \eta$.
Similarly, the space $H_\col$ is defined as the space of operators
$\xi_0 \otimes \xi$ ($\xi \in H$).

It is easy to see that, if $K$ is a subspace of $H$, then
$K_\col$ ($K_\row$) is a subspace of $H_\col$ (resp.~$H_\row$).
For simplicity of notation, we write denoted by $\row$ and $\col$,
instead of $(\ell_2)_\row$ and $(\ell_2)_\col$, respectively.
One can use matrix units to describe these spaces. We denote by
$E_{ij} \in B(\ell_2)$ the infinite matrix with $1$ on the intersection
of the $i$-th row and the $j$-th column, and zeroes elsewhere.
Then $\row$ ($\col$) is the closed linear span of $(E_{1j})_{j=1}^\infty$ (respectively, $(E_{i1})_{i=1}^\infty)$).
Below we list a few useful properties of row and column spaces. Here, $H$
and $K$ are Hilbert spaces.
\begin{enumerate}
\item
$H_\row$ and $H_\col$ are isometric to $H$ (as Banach spaces).
\item
For any $u \in B(H,K)$,
$\|u\| = \|u\|_{CB(H_\row,K_\row)} = \|u\|_{CB(H_\col,K_\col)}$.
\item
If $(\xi_i)$ is an orthonormal system in $H$, then,
for any finite sequence $(a_i)$ of elements of $M_n$,
$$
\begin{array}{lllll}
\|\sum_i a_i \otimes \xi_i\|_{M_n(H_\row)} & = &
\|(\sum_i a_i a_i^*)^{1/2}\| & = & \|\sum_i a_i a_i^*\|^{1/2} ,
\cr
\|\sum_i a_j \otimes \xi_i\|_{M_n(H_\col)} & = &
\|(\sum_i a_i^* a_i)^{1/2}\| & = & \|\sum_i a_i^* a_i\|^{1/2} .
\end{array}
$$
\item
Suppose $H^{(1)}, \ldots, H^{(n)}$ are Hilbert spaces. Then the
formal identity map $id : (H^{(1)} \oplus_2 \ldots \oplus_2 H^{(n)})_\row
\to H^{(1)}_\row \oplus \ldots \oplus H^{(n)}_\row$ is a complete
contraction, and $\|id^{-1}\|_{cb} \leq \sqrt{n}$.
The same is also true for column spaces.
\item
For any $u \in CB(H_\row, K_\col)$ or $u \in CB(H_\col, K_\row)$,
$\|u\|_2 = \|u\|_{cb}$ (here, $\| \cdot \|_2$ is the Hilbert-Schmidt norm).
\item
Duality: $H_\row^* = H_\col$, and $H_\col^* = H_\row$.
\item
For any operator space $X$
and any $u \in B(X, H_\col)$ ($u \in B(X, H_\row)$),
$\|u\|_{cb} = \|I_\col \otimes u\|_{B(\col \otimes X, \col \otimes H_\col)}$
(resp.
$\|u\|_{cb} = \|I_\row \otimes u\|_{B(\row \otimes X, \row \otimes H_\row)}$).
This result follows from the proof of Smith's lemma --
see e.g. Proposition 2.2.2 of \cite{ER}.
\item
If $H$ is separable infinite dimensional, then $H_\row$
($H_\col$) is completely isometric to $\row$ (resp.~$\col$).
\end{enumerate}

A sequence $(x_i)_{i \in I}$ in an operator space $X$ is called
{\it normalized} if $\|x_i\| = 1$ for every $i \in I$. $(x_i)$ is
said to be a {\it $c$-completely unconditional basic sequence} ($c \geq 1$)
if, for any finite sequence of matrices $(a_i)$, and any sequence of scalars
$\lambda_i \in \{ \lambda \in \C : |\lambda| \leq 1\}$,
we have $\|\sum_i \lambda_i a_i \otimes x_i\|_{M_n(X)} \leq
c \|\sum_i a_i \otimes x_i\|_{M_n(X)}$. It is easy to see that any
such sequence $(x_i)$ is linearly independent.
A $c$-completely unconditional basic sequence $(x_i) \subset X$
is called a {\it $c$-completely unconditional basis} in $X$
if $X = \span[x_i : i \in I]$ (here and below, $\span[{\mathcal{F}}]$
refers to the closed linear span of the family ${\mathcal{F}}$).

A convenient example of a normalized $1$-completely unconditional basis is
provided by an orthonormal basis in $\row$ or $\col$. For future reference,
we observe that any completely unconditional basis is ``similar to''
an orthogonal one. More precisely, suppose $X$ is an operator space,
which is isometric to a Hilbert space (this is the setting we are
concerned with in this paper). Suppose, furthermore, that $(x_i)$ is a
normalized $c$-completely unconditional basis in
an operator space $X$, isometric to a Hilbert space,
then $(x_i)$ is ``similar to'' an orthonormal basis:
for any finite sequence of scalars $(\alpha_i)$,
$$
c^2 \|\sum_i \alpha_i x_i\|^2 \geq
{\mathrm{Ave}}_\pm \|\sum_i \pm \alpha_i x_i\|^2 =
\sum_i |\alpha_i|^2 ,
$$
and similarly,
$c^{-2} \|\sum_i \alpha_i x_i\|^2 \leq \sum_i |\alpha_i|^2$.
Thus, there exists an $U : X \to \ell_2(I)$, s.t. $(U x_i)_{i \in I}$
is an orthonormal basis, and $\|U\|, \|U^{-1}\| \leq c$.

Families $(x_i)_{i \in I}$ and $(y_i)_{i \in I}$ in operator spaces
$X$ and $Y$, respectively, are called {\it $c$-equivalent} if there exists
a map $u : \span[x_i : i \in I] \to \span[y_i : i \in I]$, such that
$u x_i = y_i$, and $\|u\|_{cb} \|u^{-1}\|_{cb} \leq c$.
Two sequences are {\it equivalent} if they are $c$-equivalent,
for some $c$. Furthermore, $(x_i)_{i \in I}$ is $c$-equivalent
to a subfamily of $(y_j)_{j \in J}$ if there exists a subset
$I^\prime \subset J$, such that $|I| = |I^\prime|$, and
$(x_i)_{i \in I}$ is $c$-equivalent to $(y_j)_{j \in I^\prime}$.
Finally, $(x_i)_{i \in I}$ is {\it $c$-equivalent to $(y_i)_{i \in I}$
up to a permutation} if there exists a bijection $\pi : I \to I$
such that $(x_i)_{i \in I}$ is $c$-equivalent to $(y_{\pi(i)})_{i \in I}$.

A sequence $(x_i)_{i \in I}$ is called {\it completely unconditional}
if it is $c$-completely unconditional, for some $c$. A completely
unconditional basis, equivalence of sequences etc. are defined by
dropping the $c$, in a similar manner.

\section{Subspaces of $\row \oplus \col$: basic facts}\label{construct}

Suppose $H$ and $K$ are separable Hilbert spaces, and $A \in B(H,K)$
is a contraction. Denote by $X_\row(H,K,A)$ the subspace of
$H_\row \oplus K_\col$, spanned by $(e, Ae)$ ($e \in H$). If there is
no confusion as to the spaces $H$ and $K$, we simply write $X_\row(A)$. $X_\col(A)$ is defined in the same way.
Often, we write $X(A)$ ($X(H,K,A)$) instead of $X_\row(A)$ ($X_\row(H,K,A)$).
Note that the formal identity map
$id : H \to X(H,K,A) : \xi \mapsto \xi \oplus A \xi$ is an isometry.
Thus, we identify subspaces of $X(H,K,A)$ with those of $H$.
This identification gives meaning to the notation $A|_Y$, where
$Y \hra X(A)$.

\begin{remark}\label{connections}
Although the spaces $\row$ and $\col$ are ``simple'', the structure of
their direct sum $\row \oplus C$ is rather rich.
For instance, it was shown in \cite{Xu} that the ``operator Hilbert space''
$OH$ is a subspace of a quotient of $\row \oplus \col$
(actually, the results of that paper are much more general).
It follows from \cite{Xu} that the spaces $X(A)$ ($A \in B(\ell_2)$)
defined as above are natural ``building blocks'' of subspaces of
$\row \oplus \col$.
\end{remark}

Begin by stating a simple lemma.

\begin{lemma}\label{simple_X(A)}
Suppose $H$, $K$, and $K^\prime$ are Hilbert spaces, and $A \in B(H,K)$
is a contraction.
\begin{enumerate}
\item
Suppose $U \in B(K,K^\prime)$ is such that $\|U \xi\| = \|\xi\|$ for
any $\xi \in \ran A$. Then $X(H,K,A)$ is completely isometric to
$X(H,K^\prime,UA)$. In particular, $X(H,K,A)$ is completely isometric to
$X(H,H,|A|)$, where $|A| = (A^* A)^{1/2}$.
\item
Suppose $P_1, \ldots, P_m$ and $Q_1, \ldots, Q_m$ are families
of orthogonal mutually orthogonal projections on $H$ and $K$,
respectively, such that $\sum_k P_k = I_H$, $\sum_k Q_k = I_K$, and
$A = \sum_{k=1}^m Q_k A P_k$. Set $H_k = P_k(H)$ and $K_k = Q_k(K)$.
Then the formal identity operator $id$ from $X(H,K,A)$ to
$X(H_1,K_1,Q_1AP_1) \oplus \ldots \oplus X(H_m,K_m,Q_mAP_m)$ is
a complete contraction, and $\|id^{-1}\|_{cb} \leq \sqrt{m}$.
\item
$X(A)$ is $|\lambda|^{-1}$-completely isomorphic to $X(\lambda A)$ whenever
$0 < |\lambda| \leq 1$.
\end{enumerate}
\end{lemma}

\begin{proof}
We only establish part (2). A Gram-Schmidt orthogonalization shows that
any element $x \in M_n(H)$ can be written as
$x = x_1 + \ldots + x_n$, with $x_k = \sum_i a_{ki} \otimes \xi_{ki}$
$a_{ki} \in M_n$, and $(\xi_{ki})$ a finite orthonormal systems
in $H_k$. Then
$$
\|x\|_{M_n(X(A))} = \max \big\{
\|\sum_k \sum_i a_{ki} \otimes \xi_{ki}\|_{M_n(H_\row)} ,
\|\sum_k \sum_i a_{ki} \otimes A \xi_{ki}\|_{M_n(K_\col)} \big\} .
$$
By the basic properties of row and column spaces (listed at the end
of the previous section),
$$
\eqalign{
\max_k \|\sum_i a_{ki} \otimes \xi_{ki}\|_{M_n(H_\row)}
&
\leq
\|\sum_{k,i} a_{ki} \otimes \xi_{ki}\|_{M_n(H_\row)}
\cr
&
\leq 
\sqrt{m} \max_k \|\sum_i a_{ki} \otimes \xi_{ki}\|_{M_n(H_\row)} .
}
$$
Furthermore, the vectors $A \xi_{ki}$ belong to the mutually
orthogonal spaces $K_k$, hence
$$
\eqalign{
\max_k \|\sum_i a_{ki} \otimes A \xi_{ki}\|_{M_n(H_\col)}
&
\leq
\|\sum_{k,i} a_{ki} \otimes A \xi_{ki}\|_{M_n(H_\col)}
\cr
&
\leq 
\sqrt{m} \max_k \|\sum_i a_{ki} \otimes A \xi_{ki}\|_{M_n(H_\col)} .
}
$$
Note that
$$
\eqalign{
&
\|x_k\|_{M_n(X(H_k, Q_k, Q_kAP_k))}
\cr
&
=
\max\ \{ \|\sum_i a_{ki} \otimes \xi_{ki}\|_{M_n(H_\row)} ,
\|\sum_i a_{ki} \otimes A \xi_{ki}\|_{M_n(H_\col)} \} .
}
$$
Therefore,
$$
\eqalign{
&
\max_k \|x_k\|_{M_n(X(H_k, Q_k, Q_kAP_k))}
\cr
&
=
\max \{ \max_k \|\sum_i a_{ki} \otimes \xi_{ki}\|_{M_n(H_\row)} ,
\max_k \|\sum_i a_{ki} \otimes A \xi_{ki}\|_{M_n(H_\col)} \}
\cr
&
\leq \|x\|_{M_n(X(A))}
\cr
&
\leq
\sqrt{m} \max \{ \max_k \|\sum_i a_{ki} \otimes \xi_{ki}\|_{M_n(H_\row)} ,
\max_k \|\sum_i a_{ki} \otimes A \xi_{ki}\|_{M_n(H_\col)} \} .
}
$$
We complete the proof by noting that
$$
\|x\|_{M_n(X(H_1,K_1,Q_1AP_1) \oplus \ldots \oplus X(H_m,K_m,Q_mAP_m))} =
\max_{1 \leq k \leq m} \|x_k\|_{M_n(X(H_k, Q_k, Q_kAP_k))} .
$$
\end{proof}

Next, we establish our key tool for computing c.b.~norms.

\begin{lemma}\label{compute_norm}
Suppose $A$ and $B$ are contractions, and $T \in B(X(A), X(B))$. Then
$$
\|T\|_{cb} = \max \big\{ \|T\|, \sup \{ \|BTu\|_2 : 
u \in B(\ell_2, X(A)) , \, \|Au\|_2 \leq 1  , \, \|u\| \leq 1 \} \big\} .
$$
\end{lemma}

\begin{proof}
By definition, 
$\|T\|_{cb} = \max\{\|T\|_{CB(X(A), \row)}, \|B T\|_{CB(X(A), \col)}\}$.
To estimate the first term, note that $id : X(A) \to \row$
is a complete contraction, hence
$$
\|T\|_{CB(X(A), \row)} = \|T \circ id\|_{CB(X(A), \row)} \leq
\|T\|_{CB(\row, \row)} \|id\|_{CB(X(A), \row)} = \|T\| .
$$
However, $\|T\| \leq \|T\|_{CB(X(A), \row)}$, hence 
$\|T\| = \|T\|_{CB(X(A), \row)}$.

Next we estimate $\|B T\|_{CB(X(A), \col)}$. We know that
$$
\|B T\|_{CB(X(A), \col)} = \sup \big\{ 
\|(I_\col \otimes BT)(x)\|_{\col \otimes \col} :
x \in \col \otimes X(A) , \, \|x\|_{\col \otimes X(A)} \leq 1 \big\} .
$$
Identifying elements of $\col \otimes X(A)$ and $\col \otimes \col$
with operators from $\row$ to $X(A)$ and $\col$, respectively, we see
that
$$
\|B T\|_{CB(X(A), \col)} = \sup \big\{ 
\|BTu\|_2 : u \in CB(\row, X(A)) , \, \|u\|_{cb} \leq 1 \big\} .
$$
But
$\|u\|_{cb} = \max\{ \|u\|_{CB(\row)}, \|Au\|_{CB(\row, \col)} \} =
\max\{ \|u\|, \|Au\|_2 \}$.
\end{proof}

Next we examine the exactness of $X(A)^*$.
Recall that an operator space $X$ is called {\it exact} if
there exists $c > 0$ such that, for any finite dimensional
subspace $E$ of $X$, there exists an operator $u$ from $E$ to
a subspace $F$ of $M_n$ ($n \in \N$), such that
$\|u\|_{cb} \|u^{-1}\|_{cb} < c$. The infimum of all the
$c$'s like this is called the {\it exactness constant} of $X$.
It is easy to see that $\row \oplus \col$ is $1$-exact, hence
so is $X(A)$. The case of the dual is different.

\begin{proposition}\label{exact}
Suppose $A \in B(H,K)$ is a contraction. Then
the exactness constant of $X(A)^*$ is at least $2^{-5/2} \|A\|_2$.
In particular, $X(A)^*$ is not exact if $A$ is not Hilbert-Schmidt.
\end{proposition}

\begin{proof}
Let $c = 2^{5/2} {\mathrm{ex}}(X(A)^*)$. By Corollary 0.7 of \cite{PiS},
there exist operators $T_1 : X(A) \to \row$, $T_2 : \row \to X(A)$,
$S_1 : X(A) \to \col$, and $T_2 : \col \to X(A)$, such that
$id = S_2 S_1 + T_2 T_1$ ($id$ is the identity on $X(A)$), and
$$
\|S_1\|_{cb} \|S_2\|_{cb} + \|T_1\|_{cb} \|T_2\|_{cb} \leq
\max \{ \|S_1\|_{cb}, \|T_1\|_{cb} \} (\|S_2\|_{cb} + \|T_2\|_{cb})
\leq c
$$
By Lemma~\ref{compute_norm}, 
$\|T_2\|_{cb} \geq \|A T_2\|_2$, and a simple calculation yields
$$
\|S_2\|_{cb} \geq \|S_2\|_{CB(\col, \row)} = \|S_2\|_2 \geq \|A S_2\|_2 .
$$
Therefore,
$$
\eqalign{
\|S_1\|_{cb} \|S_2\|_{cb} + \|T_1\|_{cb} \|T_2\|_{cb}
&
\geq
\|S_1\| \|A S_2\|_2 + \|T_1\| \|A T_2\|_2
\cr
&
\geq
\|A (S_2 S_1 + T_2 T_1)\|_2 = \|A\|_2 ,
}
$$
which implies the desired estimate for $c$.
\end{proof}

\begin{corollary}\label{ci_to_row}
The space $X(A)$ is completely isomorphic to $\row$ if and only if
$A$ is Hilbert-Schmidt. Furthermore, the formal identity map
$id : \row \to X(A)$ is completely contractive, and
$\|id^{-1}\|_{cb} = \max\{1, \|A\|_2\}$.
\end{corollary}

\begin{proof}
The estimates on the c.b.~norm of $id$ and $id^{-1}$ follow from
Lemma~\ref{compute_norm}. Thus, $X(A) \coi{\max\{1,\|A\|_2\}} \row$
whenever $A$ is Hilbert-Schmidt. On the other hand, if $A$ is not Hilbert-Schmidt,
then, by Proposition~\ref{exact}, $X(A)^*$ is not exact. However,
$\row^* = \col$ is exact, thus $X(A)$ is not completely isomorphic to $\row$.
\end{proof}


We say that $A \in B(H,K)$ ($H$ and $K$ are Hilbert spaces) is
{\it diagonalizable} if the eigenvectors of $|A| = (A^* A)^{1/2}$
span $H$. Equivalently, there exist orthonormal systems $(\xi_i)$
and $(\eta_i)$ in $H$ and $K$, respectively, and a sequence
$(\lambda_i)$ of non-negative numbers, such that
$A = \sum_i \lambda_i \xi_i \otimes \eta_i$.

When working with $X(A)$, it is often convenient to have $A$
diagonalizable. While every compact operator is diagonalizable,
a non-compact one need not have this property. However, we have:

\begin{lemma}\label{diagonal}
Suppose $H$ and $K$ are separable Hilbert spaces, $A \in B(H,K)$ is
a contraction, and $\vr > 0$. Then there exists a diagonalizable contraction
$B \in B(H,K)$ such that $\|A-B\|_2 \leq \vr$, and $X(A)$ is
$(1+\vr)^2$-completely isomorphic to $X(B)$.
If $A$ is non-negative, $B$ can be selected to be non-negative, too.
\end{lemma}

\begin{proof}
By Lemma~\ref{simple_X(A)}(1), it suffices to consider the case of
$A = |A|$, and $H = K$.
By \cite{Vo}, there exists a selfadjoint diagonalizable $C \in B(H)$ such
that $\|A - C\|_2 < \vr/2$. Let $(c_i)$ be the eigenvalues of $C$, and
$(\xi_i)$ the corresponding norm $1$ eigenvectors. Define the operator $B$
by setting $B \xi_i = b_i \xi_i$, where
$$
b_i = \left\{ \begin{array}{ll}
   c_i   &   0 \leq c_i \leq 1  \cr
     0   &   c_i < 0   \cr
     1   &   c_i > 1
\end{array} \right. .
$$
We claim that $\|B - C\|_2 < \vr/2$. Indeed, let
${\mathcal{I}} = \{i : b_i \neq c_i\}$. Then
$$
\|B-C\|_2^2 = \sum_{i \in \mathcal{I}} |b_i - c_i|^2 =
\sum_{i : c_i < 0} |c_i|^2 + \sum_{i : c_i > 1} |b_i|^2 .
$$
But $\langle \xi_i, A \xi_i \rangle \in [0,1]$, hence, for
$c_i < 0$, $|\langle \xi_i, A \xi_i \rangle -
\langle \xi_i, C \xi_i \rangle| \geq |c_i|$. Similarly,
for $c_i > 1$, $|\langle \xi_i, A \xi_i \rangle -
\langle \xi_i, C \xi_i \rangle| \geq |1-c_i|$. Thus,
$$
\frac{\vr^2}{4} > \|A-C\|_2^2 =
\sum_{i,j} |\langle \xi_i, (A-C) \xi_j \rangle|^2 \geq
\sum_i |\langle \xi_i, (A-C) \xi_i \rangle|^2 \geq
\sum_{i \in \mathcal{I}} |b_i - c_i|^2 .
$$
By the triangle inequality, $\|A - B\|_2 < \vr$.
Moreover, $B$ is a non-negative contraction.

Denote the the formal identity map from $X(A)$ to $X(B)$
by $U$. It remains to show that
$\|U\|_{cb}, \|U^{-1}\|_{cb} \leq 1+\vr$.
As $U$ is an isometry, Lemma~\ref{compute_norm} implies
\begin{equation}
\|U\|_{cb} =
\sup \{ \|Bu\|_2 : 
u \in B(\ell_2, X(A)) , \, \|Au\|_2 \leq 1  , \, \|u\| \leq 1 \} .
\label{norm_U}
\end{equation}
By the triangle inequality,
$$
\|Bu\|_2 \leq \|Au\|_2 + \|(B-A)u\|_2 \leq
\|Au\|_2 + \|B-A\|_2 \leq 1 + \vr ,
$$
hence \eqref{norm_U} yields $\|U\|_{cb} \leq 1+\vr$.
$\|U^{-1}\|_{cb}$ is estimated similarly.
\end{proof}

\begin{proposition}\label{add_row}
Suppose $H$ and $K$ are separable Hilbert spaces, $B \in B(H,K)$
is a contraction, and $0 \in \sigma_{ess}(|B|)$. Then
$X(B)$ is $4\sqrt{2}$-completely isomorphic to $X(B) \oplus \row$.
\end{proposition}

\begin{proof}
We can assume $B = |B|$. By Lemma~\ref{diagonal}, there exists
a non-negative diagonalizable contraction $A$ such that $A-B$ is
Hilbert-Schmidt, and $X(A) \coi{2^{1/4}} X(B)$.
As the essential spectrum is stable under compact
perturbations, $0 \in \sigma_{ess}(A)$. Write $A =  \diag(\alpha)$,
where $\alpha = (\alpha_i)_{i \in I}$. Then $0$ is a cluster point of
the set $(\alpha_i)$. Denote the norm $1$ eigenvectors, corresponding to
$\alpha_i$, by $\xi_i$. Find $I_0 \subset I$ such that
$\sum_{i \in I_0} \alpha_i^2 < 1$. Let $I_1 = I \backslash I_0$,
and let $P_0$ and $P_1$ be the orthogonal projections onto
$H_0 = \span[\xi_i : i \in I_0]$ and
$H_1 = \span[\xi_i : i \in I_1]$, respectively.
By Lemma~\ref{simple_X(A)}, $X(A)$ is $\sqrt{2}$-completely isomorphic to
$X(H_0, H_0, P_0 A P_0) \oplus X(H_1, H_1, P_1 A P_1)$.
By Corollary~\ref{ci_to_row}, $X(H_0, H_0, P_0 A P_0)$ is
completely isometric to $\row$. Thus,
$$
\eqalign{
&
X(A) \coi{\sqrt{2}} X(H_0, H_0, P_0 A P_0) \oplus X(H_1, H_1, P_1 A P_1)
\cr
&
\coi{\sqrt{2}} \row \oplus X(H_1, H_1, P_1 A P_1)
\coi{\sqrt{2}} \row \oplus \row \oplus X(H_1, H_1, P_1 A P_1)
\cr
&
= 
\row \oplus X(H_0, H_0, P_0 A P_0) \oplus X(H_1, H_1, P_1 A P_1)
\coi{\sqrt{2}} \row \oplus X(A) .
}
$$
To summarize, $X(A) \coi{4} X(A) \oplus \row$. As $X(A) \coi{2^{1/4}} X(B)$,
we are done.
\end{proof}

\begin{corollary}\label{inf_ker}
Suppose $H$ and $K$ are separable infinite dimensional Hilbert
spaces, and a contraction $A \in B(H,K)$ satisfies $0 \in \sigma_{ess}(|A|)$.
Consider the operator $\tilde{A} = A \oplus 0$ from
$\tilde{H} = H \oplus_2 \ell_2$ to
$\tilde{K} = K \oplus_2 \ell_2$. Then $X(A) \coi{8} X(\tilde{A})$.
\end{corollary}

\begin{proof}
Let $P$ be the orthogonal projection from $\tilde{H}$ onto $H$.
onto $H$. By Lemma~\ref{simple_X(A)}(2) and Corollary~\ref{ci_to_row},
$X(\tilde{A}) \coi{\sqrt{2}} X(A) \oplus \row$. However, by
Proposition~\ref{add_row}, $X(A) \coi{4\sqrt{2}} X(A) \oplus \row$.
\end{proof}

\section{Classification of subspaces using sequences}\label{sequences}

In this section, we study the spaces $X(A)$ when $A$ is compact, and
establish a connection between such spaces and a certain family of sequences.
We denote by $\comp$ the space of compact contractions $A \in B(\ell_2)$,
which are not Hilbert-Schmidt. We denote by $\fami$ the set of all spaces
$X(A)$, with $A \in \comp$.

Start by defining the canonical basis in $X(A)$, where $A \in B(H,K)$
is a compact contraction. As $X(A) = X(|A|)$, we assume henceforth that
$A = |A|$, and $H = K$.
Let 
$(\xi_i)_{i \in I_1}$ 
be the normalized eigenvectors of $A$, corresponding to the positive
eigenvalues of $A$.
Furthermore, set $H^\prime = \ker A$, and
let $(\xi_i)_{i \in I_0}$ be an orthonormal basis in $H^\prime$
(we assume that $I_1 \cap I_0 = \emptyset$). Let $I = I_0 \cup I_1$.
The vectors $e_i = \xi_i \oplus A \xi_i \in H_\row \oplus H_\col$
($ i \in I$) span $X(A)$. Moreover, $\|e_i\| = 1$ for each $i$, and,
for any finite sequence $(a_i) \subset M_n$,
\begin{equation}
\|\sum_i a_i \otimes e_i\|_{M_n(X(A))}^2 = \max \Big\{
\big\|\sum_i a_i a_i^*\big\|, \big\|\sum_i \|A \xi_i\|^2 a_i^* a_i\big\| \Big\} .
\label{basis_general}
\end{equation}
We say that the vectors $e_i = e_i[A]$ form the {\it canonical basis}
of $X(A)$.

Now suppose $(\alpha_i)_{i \in \N}$ is a sequence of numbers in $[0,1]$.
In an effort to link operator spaces with certain sequences of scalars,
we define the operator space $\xd(\alpha) = X(\diag(\alpha))$.
To describe the operator space structure of $\xd(\alpha)$, denote by
$(\xi_i)$ the canonical orthonormal basis of $\ell_2$.
Then the vectors
$e_i(\alpha) = e_i = \xi_i \oplus \alpha_i \xi_i \in H_\row \oplus K_\col$
form a $1$-completely unconditional orthonormal basis in $\xd(\alpha)$, with
\be
\|\sum_i a_i \otimes e_i\|^2 = \max\{ \|\sum_i a_i a_i^*\| ,
\|\sum_i \alpha_i^2 a_i^* a_i\| \} 
\label{basis}
\ee
for any finite sequence of matrices $(a_i)$. We call
$(e_i(\alpha))_{i \in \N}$ the {\it canonical basis} of $\xd(\alpha)$.
The {\it formal identity} from $\xd(\alpha)$ to $\xd(\beta)$ is the
linear operator mapping $e_i(\alpha)$ to $e_i(\beta)$.

If $\alpha = (\alpha_i)_{i=1}^n$ is a finite sequence, we define
the finite dimensional space $\xd(\alpha)$ in the same way.

To reduce ourselves to working with the spaces $\xd(\alpha)$
(and hence to sequences of scalars), define, for a compact
$A \in B(H,K)$, the sequence $\alpha = {\mathbf{D}}(A)$: if $A$ has
rank $n < \infty$, let $\alpha_1 \geq \ldots \geq \alpha_n > 0$ be
the non-zero singular values of $A$, and set $\alpha_i = 0$ for $i > n$.
If the rank of $A$ is infinite, let $(\alpha_i)$ be the singular values
of $A$, listed in the non-increasing order. We have:

\begin{proposition}\label{make_diag}
If $H$ and $K$ are separable infinite dimensional Hilbert spaces,
and $A \in B(H,K)$ is a compact contraction, then $X(A)$
is $2^6$-completely isomorphic to $\xd({\mathbf{D}}(A))$.
\end{proposition}


\begin{proof}
Let $\alpha = (\alpha_i) = {\mathbf{D}}(A)$. By Lemma~\ref{simple_X(A)}(1),
we can assume that $A = |A| = (A^* A)^{1/2}$, and $H = K$.
If $A$ has finite rank, \eqref{basis_general} shows that $X(A)$ is
completely isometric to $\xd(\alpha)$. Otherwise, write
$A = \sum_i \alpha_i \xi_i \otimes \xi_i$, for some orthonormal system
$(\xi_i)_{i=1}^\infty$ in $H$. Let $P$ be the orthogonal projection
onto $K = \span[\xi_i : i \in \N]$. Set $Q = I - P$, and $L = Q(H)$.
By Lemma~\ref{simple_X(A)}(2), $X(A)$ is $\sqrt{2}$-completely isomorphic to
$X(K, K, P A P) \oplus X(L, L, Q A Q)$. Furthermore, $QAQ = 0$, hence
$X(L, L, Q A Q) = L_\row$. It is easy to see that
$X(K, K, P A P) = \xd(\alpha)$. By Proposition~\ref{add_row}, the
latter space is $4\sqrt{2}$-completely isomorphic to
$\xd(\alpha) \oplus \row$. Thus,
$$
\eqalign{
X(A)
&
\coi{\sqrt{2}}
X(K, K, P A P) \oplus X(L, L, Q A Q) \coi{4\sqrt{2}}
\xd(\alpha) \oplus \row \oplus L_\row
\cr
&
\coi{\sqrt{2}}
\xd(\alpha) \oplus (\ell_2 \oplus L)_\row =
\xd(\alpha) \oplus \row \coi{4\sqrt{2}} \xd(\alpha) .
}
$$
\end{proof}

We will also use a related observation.

\begin{lemma}\label{almost_alpha}
For every compact contraction $A \in B(\ell_2)$, and every $\vr > 0$,
there exists $\alpha \in c_0$ such that $X(A) \coi{1+\vr} \xd(\alpha)$.
\end{lemma}

\begin{proof}
Indeed, let $\beta = {\mathbf{D}}(A)$. If $A$ is finite rank, then $X(A)$
is completely isometric to $\xd(\beta)$. If $\rank A = \infty$, assume
(by Lemma~\ref{simple_X(A)}) that $A = |A|$. Then
$\beta_1 \geq \beta_2 \geq > 0$ is the list of all positive eigenvalues of $A$.
Denote the corresponding norm $1$ eigenvectors by $\xi_i$. Let
$H = \span[\xi_i : i \in \N]$, and $K = \ker A$. Clearly,
$K$ and $H$ are mutually orthogonal subspaces of $\ell_2$. Let
$(\eta_j)_{j \in J}$ be the orthogonal basis of $K$. Find a
sequence $(\gamma_j)_{j \in J}$ of positive numbers, satisfying
$\sum_{j \in J} \gamma_i^2 < \vr^2$. Consider the compact contraction
$\tilde{A} \in B(\ell_2)$, defined by $\tilde{A} \xi_i = \beta_i \xi_i$,
and $\tilde{A} \eta_j = \gamma_j \eta_j$. Then $\|A - \tilde{A}\|_2 < \vr$.
By \eqref{basis_general}, the formal identity map
$id : X(\tilde{A}) \to X(A)$ is a complete contraction, and
$\|id^{-1}\|_{cb} < 1 + \vr$. We complete the proof by identifying
$X(\tilde{A})$ with the space $\xd(\alpha)$, where the sequence
$\alpha = (\alpha_i)$ is the ``join'' of the sequences $\beta$
and $\gamma$ (that is, any number $c \in [0,1]$ occurs in
$\alpha$ as many times as it occurs in the sequences $\beta$
and $\gamma$ combined).
\end{proof}



Now denote by $\is$ the set of all sequences $(\alpha_i)_{i \in \N}$
satisfying $1 \geq \alpha_1 \geq \alpha_2 \geq \ldots \geq 0$, and
$\lim_i \alpha_i = 0$.
The rest of this section is devoted to the spaces $\xd(\alpha)$ ($\alpha \in \is$).
We translate the relations between sequences $\alpha, \beta \in \is$
to relations between the corresponding spaces $\xd(\alpha)$ and $\xd(\beta)$.
We say that the sequence $\alpha$ {\it dominates} $\beta$
($\alpha \succ \beta$) if there exists a set $S \subset \N$) and
$K > 0$ s.t. $\sum_{i \in S} \beta_i^2 < \infty$, and
$K \alpha_i \geq \beta_i$ for any $i \notin S$.
We say $\alpha$ is {\it equivalent} to $\beta$
($\alpha \sim \beta$) if $\alpha \succ \beta$, and
$\beta \succ \alpha$.

Clearly, the relation $\succ$ is reflexive and transitive.
The relation $\sim$ is, in addition to this, symmetric.
For instance, to establish the transitivity of $\succ$,
suppose $\alpha \succ \beta$, and $\beta \succ \gamma$, and
show that $\alpha \succ \gamma$. Note that there exist
sets $S_1$ and $S_2$, and constants $K_1$ and $K_2$, s.t.
$\beta_i \leq K_1 \alpha_i$ for $i \notin S_1$,
$\gamma_i \leq K_2 \beta_i$ for $i \notin S_2$,
$\sum_{i \in S_1} \beta_i^2 < \infty$, and
$\sum_{i \in S_2} \gamma_i^2 < \infty$. Let
$S = S_1 \cup S_2$, and $K = K_1 K_2$. Then
$\gamma_i \leq K \alpha_i$ for $i \notin S$.
Moreover,
$$
\sum_{i \in S} \gamma_i^2 = 
\sum_{i \in S_2} \gamma_i^2 +
\sum_{i \in S_1 \backslash S_2} \gamma_i^2 \leq
\sum_{i \in S_2} \gamma_i^2 +
\sum_{i \in S_1} K_2^2 \beta_i^2 < \infty ,
$$
which is what we need. The other properties are proved in a similar fashion.

\begin{proposition}\label{equiv_seq}
For $\alpha, \beta \in \is$, $\alpha \sim \beta$ if and only if there
exists a set $S \subset \N$ and a constant $K$ s.t.
$\sum_{i \in S} (\alpha_i^2 + \beta_i^2) < \infty$, and
$K^{-1} \alpha_i \leq \beta_i \leq K \alpha_i$ for $i \notin S$.
\end{proposition}

\begin{proof}
If $S$ and $K$ with the properties described above exist, then
they witness the fact that $\alpha \prec \beta$ and $\alpha \succ \beta$,
hence $\alpha \sim \beta$. Conversely, suppose $\alpha \prec \beta$ and
$\alpha \succ \beta$. Then there exist a constant $K$, and sets $S_1$
and $S_2$, s.t. $\alpha_i \leq K \beta_i$ for $i \notin S_1$,
$\beta_i \leq K \alpha_i$ for $i \notin S_2$,
$\sum_{i \in S_1} \alpha_i^2 < \infty$, and
$\sum_{i \in S_2} \beta_i^2 < \infty$.
By reducing $S_1$ and $S_2$ further, we can assume that
$\alpha_i > K \beta_i$ for each $i \in S_1$, and
$\beta_i > K \alpha_i$ for each $i \in S_2$.
Then $\sum_{i \in S_1} \beta_i^2 < \infty$, and
$\sum_{i \in S_2} \alpha_i^2 < \infty$.
Therefore, $S = S_1 \cup S_2$ has the required properties.
\end{proof}


The main result of this section is:

\begin{theorem}\label{classify}
For $\alpha, \beta \in \is$, $\xd(\alpha)$ embeds completely isomorphically
into $\xd(\beta)$ if and only if $\alpha \prec \beta$.
\end{theorem}

From this we immediately obtain

\begin{corollary}\label{classify_cor}
Suppose $\alpha, \beta \in \is$.
The following three statements are equivalent.
\begin{enumerate}
\item
$\xd(\alpha)$ is completely isomorphic to $\xd(\beta)$.
\item
$\xd(\alpha)$ embeds completely isomorphically into $\xd(\beta)$,
and vice versa.
\item
$\alpha \sim \beta$.
\end{enumerate}
\end{corollary}

\begin{proof}
(1) $\Rightarrow$ (2) is trivial, while (2) $\Rightarrow$ (3) follows
from Theorem~\ref{classify} and Proposition~\ref{equiv_seq}.
To establish (3) $\Rightarrow$ (1), suppose $\alpha \sim \beta$.
By Proposition~\ref{equiv_seq}, there exists a set $\ii$ and $K > 0$ s.t.
$\sum_{i \in \ii} (\alpha_i^2 + \beta_i^2) < \infty$, and
$K^{-1} \alpha_i \leq \beta_i \leq K \alpha_i$ for any $i \notin \ii$.
By Corollary~\ref{ci_to_row}, the spaces
$E_\alpha = \span[e_i : i \in \ii] \hra \xd(\alpha)$ and
$E_\beta = \span[e_i : i \in \ii] \hra \xd(\beta)$ are completely isomorphic
to $\row$. By \eqref{basis}, the formal identity map from
$F_\alpha = \span[e_i : i \notin \ii] \hra \xd(\alpha)$ to
$F_\beta = \span[e_i : i \notin \ii] \hra \xd(\beta)$
is a complete isomorphism.
By Lemma~\ref{simple_X(A)}(2), $\xd(\alpha) \simeq E_\alpha \oplus F_\alpha$,
and $\xd(\beta) \simeq E_\beta \oplus F_\beta$.
Therefore, the formal identity map from $\xd(\alpha)$ to $\xd(\beta)$ is
a complete isomorphism.
\end{proof}

The proof of Theorem~\ref{classify} follows from the next two lemmas.

\begin{lemma}\label{embeds}
If $\alpha, \beta \in \is$ and $\beta \prec \alpha$, then $\xd(\beta)$
embeds completely isomorphically into $\xd(\alpha)$.
\end{lemma}

\begin{proof}
By Corollary~\ref{inf_ker}, $\xd(\alpha)$ is completely isomorphic to
$X(A)$, where $A \in B(\ell_2 \oplus_2 \ell_2)$ is defined by
$A = \diag(\alpha) \oplus 0$.
Let $(\xi_i)$ and $(\xi_i^\prime)$ be
orthonormal bases in the first and second copies of $\ell_2$, respectively.
Then $X(A)$ is the closed linear span of the vectors
$\xi_i \oplus \alpha_i \xi_i$ and $\xi_i^\prime \oplus 0$ ($i \in \N$).
Find a sequence $\phi_i \in [0, \pi/2]$ s.t.
$\beta_i = \cos \phi_i \cdot \alpha_i$.
Define an orthonormal system
$\eta_i = \cos \phi_i \xi_i + \sin \phi_i \xi_i^\prime$.
For $i \in \N$ consider
$$
f_i = \eta_i \oplus \beta_i \xi_i =
\cos \phi_i (\xi_i \oplus \alpha_i \xi_i) +
\sin \phi_i (\xi_i^\prime \oplus 0) .
$$
Then $f_i \in X(A)$, and $\span[f_i : i \in \N] = \xd(\beta)$.
\end{proof}

\begin{lemma}\label{dominate}
Suppose $\alpha, \beta \in \is$, $E$ is a subspace of $\xd(\alpha)$,
and a completely bounded map $U : E \to \xd(\beta)$ has bounded
inverse (in the terminology of \cite {OiR}, $E$ is completely
semi-isomorphic to $\xd(\beta)$). Then $\beta \prec \alpha$.
\end{lemma}

\begin{proof}
We rely heavily on Wielandt's Minimax Theorem (\cite[Theorem III.6.5]{Bh}):
if $c_1 \geq c_2 \geq \ldots \geq 0$ are eigenvalues
of positive compact operator $T$, then, for any finite increasing
sequence $i_1 < \ldots < i_k$ of positive integers,
$$
\sum_{j=1}^k c_{i_j} = \sup_{E_1 \hra \ldots \hra E_k}
\min_{x_j \in E_j, \, (x_j) \, {\mathrm{orthonormal}}}
\langle T x_j , x_j \rangle ,
$$
where the supremum is taken over all subspaces $E_1 \hra \ldots \hra E_k$
of the domain of $T$, with $\dim E_j = i_j$ for $1 \leq j \leq k$.
Actually, the theorem is stated in \cite{Bh} for operators on finite 
dimensional spaces, but a generalization to compact operators is easy
to obtain. Applying the above identity to $T = S^* S$, where $S$ is a
compact operator with singular numbers $s_1 \geq s_2 \geq \ldots \geq 0$,
we obtain:
\be
\sum_{j=1}^k s_{i_j}^2 = \sup_{E_1 \hra \ldots \hra E_k, \dim E_j = i_j}
\min_{x_j \in E_j, \, (x_j) \, {\mathrm{orthonormal}}}
\|S x_j\|^2 .
\label{minimax}
\ee

In our situation, assume $\|U\|_{cb} = 1$. Let $c = \|U^{-1}\|$,
$A = \diag(\alpha_i)$, $B = \diag(\beta_i)$, $B^\prime = B U$.
Denote the singular values of $B^\prime$ by $(\beta_i^\prime)$. Clearly,
$\beta_i^\prime \leq \beta_i \leq c \beta_i^\prime$
for every $i \in \N$.

Pick an orthonormal system $(x_j)_{j=1}^k$ in $E$ (the domain of $U$).
Let $u$ be the formal identity from $\row^k = (\ell_2^k)_\row$
(the $k$-dimensional row space) to $\span[x_j : 1 \leq j \leq k]$.
Then (compare with the proof of Lemma~\ref{compute_norm})
$$
\|u\|_{cb}^2 = \max\{1, \|A u\|_2^2\} \leq
1 + \sum_{j=1}^k \|A x_j\|^2 ,
$$
and (since $U$ is a complete contraction)
$$
\|u\|_{cb}^2 \geq \|U u\|_{cb}^2 \geq \|B U u\|_2^2 =
\sum_{j=1}^k \|B^\prime x_j\|^2 .
$$
Thus, $\sum_{j=1}^k \|B^\prime x_j\|^2 \leq 
\sum_{j=1}^k \|A x_j\|^2 + 1$
for any orthonormal family $(x_j)_{j=1}^k$.
By \eqref{minimax},
\be
1 + \sum_{j=1}^k \alpha_{i_j}^2 \geq 
\sum_{j=1}^k \beta_{i_j}^{\prime 2} 
\label{fin_sum}
\ee
for any $i_1 < \ldots < i_k$ (indeed, when computing
$\sum_{j=1}^k \alpha_{i_j}^2$, we are taking the supremum over a larger
family of subspaces $(E_j)$, than when we are computing
$\sum_{j=1}^k \beta_{i_j}^{\prime 2}$).

Now let $\ii = \{i \in \N : \beta_i^\prime > 2 \alpha_i\}$.
Then $\sum_{i \in \ii} \beta_i^{\prime 2} \leq 2$.
Indeed, otherwise there exists a sequence $i_1 < \ldots < i_k$ of
elements of $\ii$ s.t. $C = \sum_{j=1}^k \beta_{i_j}^{\prime 2} > 2$.
Then, by (\ref{fin_sum}), $\sum_{j=1}^k \alpha_{i_j}^2 \geq C - 1$.
On the other hand, $\sum_{j=1}^k \alpha_{i_j}^2 \leq C/2$,
a contradiction.
\end{proof}

Developing the ideas of this proof, we obtain:

\begin{theorem}\label{same_basis}
Suppose $\alpha, \beta \in \is$, and there exists an isomorphism
$U : \xd(\alpha) \to \xd(\beta)$ with $\|U\|_{cb}, \|U^{-1}\|_{cb} \leq C$.
Then the formal identity map $id : \xd(\alpha) \to \xd(\beta)$
satisfies $\|id\|_{cb}, \|id^{-1}\|_{cb} < 4 C^2$.
\end{theorem}

\begin{proof}
As in the proof of Lemma~\ref{dominate}, let $A = \diag(\alpha_i)$,
$B = \diag(\beta_i)$, and $B^\prime = B U$.
By Lemma~\ref{compute_norm},
$\|B^\prime u\|_2 \leq C \max\{ \|A u\|_2, \|u\|\}$
for any $u : \ell_2 \to \xd(\alpha)$.
Denote the singular numbers of $B^\prime$ by $(\beta_i^\prime)$,
and note that $\beta_i/C \leq \beta_i^\prime \leq C \beta_i$
for every $i$.
Reasoning as in the proof of Lemma~\ref{dominate}, we see that
$$  
C^2(1 + \sum_{j=1}^k \alpha_{i_j}^2) \geq 
\sum_{j=1}^k \beta_{i_j}^{\prime 2} 
$$  
Let $\ii = \{i : \beta_i^\prime > 2C \alpha_i\}$.
As in the preceding proof,
$\sum_{i \in \ii}\beta_i^{\prime 2} < 2 C^2$.
Therefore, $\sum_{i \in \ii} \beta_i^2 < 2 C^4$, and
$\beta_i \leq 2 C^2 \alpha_i$ for $i \notin \ii$.


Next we show that
$\|id : \xd(\alpha) \to \xd(\beta)\|_{cb} < 4 C^2$.
As before, denote the canonical bases in $\xd(\alpha)$ and $\xd(\beta)$
by $(e_i(\alpha))_{i \in \N}$ and $(e_i(\beta))_{i \in \N}$, respectively.
By \eqref{basis}, $Y_1 = \span[e_i(\alpha) : i \in \ii]$ and
$Y_0 = \span[e_i(\alpha) : i \notin \ii]$ are completely contractively
complemented subspaces of $\xd(\alpha)$. Moreover,
$$
\eqalign{
\|id|_{Y_1}\|_{cb} \leq \|V : (Y_1)_\row \to \xd(\beta)\|_{cb}
&
=
\max\{1, \|B V\|_{cb}\}
\cr
&
\leq
\max\big\{1, \big( \sum_{i \in \ii} \beta_i^2 \big)^{1/2} \big\} < 2 C^2 
}
$$
(here, $V$ is the formal identity from $(Y_1)_\row$ to 
$\span[e_i(\beta) : i \in \ii]$).
By Lemma~\ref{compute_norm} and \eqref{basis}, $\|id|_{Y_0}\|_{cb} < 2 C^2$.
Therefore,
$$
\|id : \xd(\alpha) \to \xd(\beta)\|_{cb} \leq 
\|id|_{Y_1}\|_{cb} + \|id|_{Y_0}\|_{cb} < 4 C^2 .
$$
The norm of $id : \xd(\beta) \to \xd(\alpha)$ is computed the same way.
\end{proof}

\section{Completely unconditional bases}\label{bases}

In this section, we further investigate bases in spaces $X(A)$,
where $A \in B(H,K)$ is a compact contraction.
A subspace $E$ of $X(A)$ is isometric to $X(A|_E)$.
As $A|_E$ is a compact contraction,
\eqref{basis_general} implies that $E$ has a $1$-completely unconditional
basis. The key result of this section is Proposition
\ref{unique_basis},
establishing the uniqueness of a completely unconditional basis in $E$
(the existence of such a basis has been established by
Proposition \ref{make_diag}).
We also show that the canonical basis (and therefore, every completely
unconditional basis) in a completely complemented subspace of $\xd(\alpha)$
is equivalent to a subsequence of the canonical basis of $\xd(\alpha)$.
Moreover, there exists $\alpha \in \is$ such that the canonical basis
in every complemented subspace of $\xd(\alpha)$ is equivalent to a
subsequence of the canonical basis of $\xd(\alpha)$ (Theorem~\ref{subbasis}).
In general, the last statement need not be true (Remark~\ref{not_subbasis}).

First we show that any completely unconditional basic sequence in an
$X(A)$ space corresponds to a canonical basis of $\xd(\beta)$,
for some $\beta$.

\begin{proposition}\label{unconditional}
Suppose $\ii$ is a finite or countable set, and
$(e_i^\prime)_{i \in \ii}$ is a a $C$-completely unconditional
basic sequence in $X(A)$ ($A \in B(\ell_2)$ is a contraction,
not necessarily compact).
Let $Y = \span[e_i^\prime : i \in \ii]$, and define the sequence
$\beta = (\beta_i)$ by setting $\beta_i = \|A e^\prime_i\|$.
Consider the operator $T : Y \to \xd(\beta) : e_i^\prime \mapsto e_i$,
where $(e_i)$ is the canonical basis of $\xd(\beta)$. Then
$\|T\|_{cb} \leq C$ and $\|T^{-1}\|_{cb} \leq C^2$.
\end{proposition}

\begin{proof}
As noted in Section~\ref{op_sp},
\be
C^{-2} \|\sum \alpha_i e_i^\prime\|^2 \leq \sum |\alpha_i|^2 \leq
C^2 \|\sum \alpha_i e_i^\prime\|^2
\label{squares}
\ee
for any finite sequence of scalars $(\alpha_i)$.
Thus, $\|T\|, \|T^{-1}\| \leq C$.

Let $B = \diag(\beta)$ (note that $B = B^*$).
By Lemma~\ref{compute_norm}, it suffices to show that
\be
\|B T u\|_2 \leq C \max \{ \|A u\|_2, \|u\| \}
\label{equiv1}
\ee
for any $u : \ell_2 \to Y = X(A|_Y)$, and
\be
\|A T^{-1} u\|_2 \leq C^2 \max \{ \|B u\|_2, \|u\| \}
\label{equiv2}
\ee
for any $u : \ell_2 \to \xd(\beta)$.
By Lemma~\ref{compute_norm}, the complete unconditionality of
$(e_i^\prime)$ implies:
\be
\begin{array}{ll}
\|A \Lambda u\|_2
&
\leq C \max \{ \|A u\|_2, \|u\| \} ,
\\
\|A u\|_2
&
\leq C \max \{ \|A \Lambda u\|_2, \|u\| \}
\end{array}
\label{uncond}
\ee
whenever $\Lambda = \diag(\lambda_i)$ (that is,
$\Lambda e_i^\prime = \lambda_i e_i^\prime$), with
$\lambda_i = \pm 1$ for each $i$, and $u \in B(\ell_2, Y)$.
Note that $\|A u\|_2^2 = \tra (A^* A v)$, where $v = u u^*$.
Therefore, \eqref{uncond} is equivalent to
\be
\begin{array}{ll}
\tra (\Lambda^* A^* A \Lambda v) 
&
\leq C^2 \max \{ \tra(A^* A v) , 1 \} , \\
\tra (A^* A v) 
&
\leq C^2 \max \{ \tra(\Lambda^* A^* A \Lambda v) , 1 \}
\end{array}
\label{uncond_tr}
\ee
whenever $v \geq 0$ and $\|v\| = 1$.
But
$\langle \Lambda^* A^* A \Lambda e_i^\prime, e_j^\prime \rangle =
\lambda_i \overline{\lambda_j}
\langle A e_i^\prime, A e_j^\prime \rangle$.
Averaging over $\lambda_i = \pm 1$ for each $i$, we see that
${\mathrm{Ave}}_\Lambda \Lambda^* A^* A \Lambda = T^* B^2 T$
(since $\langle T^* B^2 T e_i^\prime, e_j^\prime \rangle =
\delta_{ij} \langle A e_i^\prime, A e_j^\prime \rangle$,
where $\delta_{ij}$ is Kronecker's delta).
Therefore, by \eqref{uncond_tr},
$$
\eqalign{
\tra (T^* B^2 T v)
&
=
{\mathrm{Ave}}_\Lambda \tra (\Lambda^* A^* A \Lambda v)
\cr
&
\leq
\sup_\Lambda \tra (\Lambda^* A^* A \Lambda v) \leq
C^2 \max\{ \tra (A^* A v) , 1 \}
}
$$
whenever $v$ is a positive contraction. Thus, for any
contraction $u$,
$$
\eqalign{
\|B T u\|_2^2
&
=
\tra (T^* B^2 T u u^*)
\cr
&
\leq
C^2 \max\{ \tra (A^* A u u^*) , 1 \}
= C^2 \max \{ \|A u\|_2, \|u\| \}^2 ,
}
$$
which proves (\ref{equiv1}). 

To establish (\ref{equiv2}), we show that,
for every $w : \ell_2 \to Y$, we have
\begin{equation}
\|A w\|_2 \leq C \max \{ \|B T w\|_2, \|w\| \}
\label{AwB}
\end{equation}
Indeed, let $w = T^{-1} u$. Then, by \eqref{AwB} and \eqref{squares},
$$
\eqalign{
\|A T^{-1} u\|_2
&
=
\|A w\|_2 \leq C \max \{ \|B T w\|_2, \|w\| \}
\cr
&
\leq
C \max \{ \|B u\|_2, C \|u\| \} \leq C^2 \max \{ \|B u\|_2, \|u\| \} .
}
$$
Moreover, \eqref{AwB} is equivalent to the following:
for any non-negative, norm one $v \in B(\xd(\alpha))$,
we have 
\be
C^2 \max\{ \tra (T^* B^2 T v) , 1 \} \geq \tra (A^* A v) .
\label{C_prime_C}
\ee
As we have established before, 
$\tra (T^* B^2 T v) = 
{\mathrm{Ave}}_\Lambda \tra (\Lambda^* A^* A \Lambda v)$.
By the linearity and positivity of the trace,
$$
0 \leq \inf_\Lambda \tra (\Lambda^* A^* A \Lambda v) \leq
\tra (T^* B^2 T v) \leq
\sup_\Lambda \tra (\Lambda^* A^* A \Lambda v) .
$$
Thus, for some $\Lambda$,
$\tra (\Lambda^* A^* A \Lambda v) \leq \tra (T^* B^2 T v)$.
By (\ref{uncond_tr}),
$$
\tra(A^* A v) \leq C^2 \max \{ \tra (\Lambda^* A^* A \Lambda v)  , 1 \}
\leq C^2 \max \{ \tra (T^* B^2 T v)  , 1 \} ,
$$
which implies (\ref{C_prime_C}).
\end{proof}

\begin{proposition}\label{unique_basis}
For $\alpha \in c_0$, the completely unconditional basis in $\xd(\alpha)$
is unique (up to permutative equivalence). More precisely: if $(g_i)$
is a $C$-completely unconditional basis in $\xd(\alpha)$, then
it is $16 C^{11}$-equivalent (up to a permutation) to the
canonical basis in $\xd(\alpha)$.
\end{proposition}

\begin{proof}
Let $(e_i(\alpha))$ be the canonical basis of $\xd(\alpha)$.
Set $\beta_i = \|A g_i\|$, and let $(e_i(\beta))$ the canonical
basis of $\xd(\beta)$. By Proposition~\ref{unconditional}, the map
$T : \xd(\alpha) \to \xd(\beta) : g_i \mapsto e_i(\beta)$ satisfies
$\|T\|_{cb} \leq C$, $\|T^{-1}\|_{cb} \leq C^2$.
By Theorem~\ref{same_basis}, $id : \xd(\beta) \to \xd(\alpha)$
satisfy $\|id\|_{cb}, \|id^{-1}\|_{cb} < 4 C^4$.
Thus, the operator $U = id \circ T$ is a complete isomorphism on
$\xd(\alpha)$, with $U g_i = e_i$, $\|U\|_{cb} < 4 C^5$, and
$\|U^{-1}\|_{cb} < 4 C^6$.
\end{proof}

\begin{corollary}\label{compl_sub}
Suppose $\alpha \in \is$, and $Y$ is a $C$-completely complemented
subspace of $\xd(\alpha)$. Then $Y$ is $2^6 C^2$-completely isomorphic
to $\xd(\alpha^\prime)$, where $\alpha^\prime$ is a subsequence of $\alpha$.
\end{corollary}

\begin{proof}
Let $P$ be a projection from $\xd(\alpha)$ onto $Y$, with
$\|P\|_{cb} \leq C$. Then $\xd(\alpha)$ is $2C$-completely isomorphic
to $Y \oplus Z$, where $Z = \ker P$. By Lemma~\ref{almost_alpha},
$Y$ and $Z$ are $\sqrt{2}$-completely isomorphic to $\xd(\beta)$ and $\xd(\beta^\prime)$, respectively,
where $\beta$ and $\beta^\prime$ belong to $\is$.
By Lemma~\ref{simple_X(A)}(2), $\xd(\alpha)$ is $4 C$-completely isomorphic
to $\xd(\gamma)$, where $\gamma = (\gamma_i) \in \is$ is the ``join'' of
$\beta = (\beta_i)$ and $\beta^\prime = (\beta^\prime_i)$. More precisely,
the sequence $\gamma$ has the property that, for every $c \in [0,1]$,
$$
|\{ i : \gamma_i = c \}| = |\{ i : \beta_i = c \}| +
|\{ i : \beta^\prime_i = c \}| .
$$
Denoting the canonical basis of $\xd(\gamma)$ by $(e_i(\gamma))$,
we see that $\xd(\beta) = \span[e_i(\gamma) : i \in I]$, for some
infinite set $I \subset \N$. By Theorem~\ref{same_basis}, the formal
identity $id : \xd(\gamma) \to \xd(\alpha) : e_i(\gamma) \mapsto e_i(\alpha)$
satisfies $\|id\|_{cb}, \|id^{-1}\|_{cb} < 8 C$. In particular,
$\xd(\beta)$ $2^6 C^2$-completely isomorphic to
$\span[e_i(\alpha) : i \in I]$.
\end{proof}

\begin{remark}\label{oi_lms}
By \cite{Oi}, the completely unconditional basis in $\bsp$ is unique
up to a permutation.
\end{remark}

\begin{remark}\label{not_subbasis}
In general, the canonical basis of a subspace of $\xd(\alpha)$
($\alpha \in \is$) need not be equivalent to a subsequence of
the canonical basis of $\xd(\alpha)$. For instance,
suppose the sequence $\alpha = (\alpha_i)$ and $\beta = (\beta_i)$ are
defined by setting $\alpha_i = 2^{-n^2}$, $\beta_i = 2^{-n^2-n}$ for
$4^{n^2} \leq i < 4^{(n+1)^2}$ ($n \in \{0\} \cup \N$). By
Lemma~\ref{embeds}, $\xd(\beta)$ embeds completely isomorphically
into $\xd(\alpha)$. However, $(e_i(\beta))$ (the canonical basis of
$\xd(\beta)$) is not equivalent to any subsequence of
the canonical basis $(e_i(\alpha))$ of $\xd(\alpha)$.
Indeed, suppose, for the sake of achieving a contradiction, that
there exists a complete isomorphism $T$ from $\xd(\beta)$ to a subspace
of $\xd(\alpha)$, mapping $e_i(\beta)$ to $e_{k_i}(\alpha)$.
Fix $n \in \N$ with
$\max\{\|T\|_{cb}, \|T^{-1}\|_{cb}\} < 2^{n/2}$. Consider the sets
$$
\begin{array}{lll}
I_n  &  =  &  \{ 4^{n^2} \leq i < 4^{(n+1)^2} : k_i < 4^{(n+1)^2} \} , \cr
J_n  &  =  &  \{ 4^{n^2} \leq i < 4^{(n+1)^2} : k_i \geq 4^{(n+1)^2} \} .
\end{array}
$$
By Pigeon-Hole Principle, with $I_n$ or $J_n$ has the cardinality
grater than $4^{n^2+2n}$. If $|I_n| > 4^{n^2+2n}$, consider
$$
x = \sum_{i \in I_n} E_{i1} \otimes e_i(\beta) \in
M_{4^{(n+1)^2}}(\xd(\beta))
$$
(recall that $E_{i1}$ is the ``matrix unit'' with $1$ on the intersection of
the first column and the $i$-th row, and zeroes everywhere else).
By \eqref{basis},
$$
\|x\|^2_{M_{4^{(n+1)^2}}(\xd(\beta))} =
\max\{ 1, |I_n| \cdot 2^{-n^2-n} \} = |I_n| \cdot 2^{-n^2-n} .
$$
However, by \eqref{basis} again,
$$
\|(I_{M_{4^{(n+1)^2}}} \otimes T)x\|^2_{M_{4^{(n+1)^2}}(\xd(\alpha))} \geq
\max\{ 1, |I_n| \cdot 2^{-n^2-n} \} = |I_n| \cdot 2^{-n^2} ,
$$
yielding $\|T\|_{cb}^2 \geq 2^{-n}$. If $J_n > 4^{n^2+2n}$, consider
$$
x = \sum_{i \in J_n} E_{i1} \otimes e_i(\beta) \in
M_{4^{(n+1)^2}}(\xd(\beta)) .
$$
As before, $\|x\|^2 = |J_n| \cdot 2^{-n^2-n}$, while
$$
\|(I_{M_{4^{(n+1)^2}}} \otimes T)x\|^2_{M_{4^{(n+1)^2}}(\xd(\alpha))} \geq
\max\{ 1, |I_n| \cdot 2^{-(n+1)^2} \} = |J_n| \cdot 2^{-(n+1)^2} ,
$$
hence $\|T^{-1}\|_{cb}^2 \geq 2^{n+1}$. Thus,
$\max\{\|T\|_{cb}, \|T^{-1}\|_{cb}\} \geq 2^{n/2}$,
which yields the desired contradiction.
\end{remark}

In certain situations, the canonical basis for every subspace of
$\xd(\alpha)$ is equivalent to a subsequence of the canonical basis
of $\xd(\alpha)$.


\begin{theorem}\label{subbasis}
For any $a > 1$ there exists $\alpha \in \is \backslash \ell_2$ such that
any subspace $Y$ of $\xd(\alpha)$ (finite or infinite dimensional) has a
$1$-completely unconditional basis, $a$-equivalent to a subsequence
of the canonical basis of $\xd(\alpha)$.
\end{theorem}

Combining this result with Proposition~\ref{unique_basis}, we obtain

\begin{corollary}\label{cor_subbasis}
There exists $\alpha \in \is \backslash \ell_2$ such that
any $C$-completely unconditional basic sequence in $\xd(\alpha)$
is $A C^B$-equivalent to a subsequence of the canonical basis
of $\xd(\alpha)$ (here, $A$ and $B$ are positive).
\end{corollary}

\begin{proof}[Proof of Theorem~\ref{subbasis}]
Assume $a < 2$. 
Pick a sequence of integers $1 = N_0 < N_1 < \ldots$, s.t.
$N_k > 2 N_{k-1}$ for each $k$. Define a sequence $\alpha = (\alpha_i)$
by setting $\alpha_{2i} = a^{-k}$ for $N_k \leq i < N_{k+1}$,
$\alpha_{2i-1} = 0$ for any $i \in \N$.
Clearly, $\alpha \in c_0 \backslash \ell_2$.
Let $A = \diag(\alpha)$. For a subspace $Y$ of $X(A)$, let
$\beta = (\beta_i)$ is the sequence of singular values of $A|_Y$.
Define the set $\ii_1$ by
setting $\ii_1 = \{1, \ldots, M\}$ if $\rank (A|_Y) = M < \infty$,
and $\ii_1 = \N$ if $\rank (A|_Y) = \infty$.
We can also assume that the elements of $(\beta_i)_{i \in \ii_1}$
are listed in the non-increasing order. Then
$\beta_i \leq \alpha_{2i}$ for each $i$.

Denote the normalized eigenvectors of $(A|_Y)^* A|_Y$, corresponding to
the eigenvalues $\beta_i$ ($i \in \ii_1$), by $\eta_i$.
Furthermore, find the vectors $(\eta_i)_{i \in \ii_0}$, forming
an orthonormal basis in $\ker (A|_Y)$. For $i \in \ii_0$, set $\beta_i = 0$.
Let $\ii = \ii_1 \cup \ii_0$ (we assume that this union is disjoint).
Then the family $(\eta_i)_{i \in \ii}$ is the canonical basis for $Y$.

For each positive integer $k$, let $M_k$ be the smallest value of $i$
s.t. $\beta_i \leq a^{-k}$. Set $M_0 = 1$. In this notation,
$a^{1-k} \geq \beta_i > a^{-k}$ iff $M_{k-1} \leq i < M_k$.
As noted above, $\beta_i \leq \alpha_{2i}$, hence $M_k \leq N_k$.
By our choice of the sequence $(N_k)$,
$$
M_k - M_{k-1} < M_k \leq N_k \leq N_{k+1} - N_k .
$$
Thus, there exists an injective map $\pi : \ii \to \N$ s.t.
$\pi(\ii_0) \subset \{2i-1 : i \in \N\}$, and
$\pi([M_{k-1}, M_k)) \subset \{2i : i \in [N_k, N_{k+1})\}$
for each $k \in \N$. For $i \in \ii_0$, $\alpha_{\pi(i)} = \beta_i = 0$,
while for $i \in \ii_1$,
$a \alpha_{\pi(i)} = a^{1-k} \geq \beta_i > a^{-k} = \alpha_{\pi(i)}$.
Define the operator $T : Y \to \span[e_{\pi(i)} : i \in \ii] \hra \xd(\alpha)$,
defined by $T \xi_i = e_{\pi(i)}$. By \eqref{basis_general},
$T$ is a complete contraction, and $\|T^{-1}\|_{cb} \leq a$.
\end{proof}

\begin{remark}\label{complem_subsp}
The proof of Theorem~\ref{subbasis} shows that any subspace of $\xd(\alpha)$
is $a$-completely isomorphic to a completely contractively complemented subspace
of $\xd(\alpha)$. Nevertheless, $\xd(\alpha)$ contains subspaces which are not
completely complemented. To construct them, find a sequence $(\beta_i)$
such that $1 \geq \beta_1 \geq \beta_2 \geq \ldots > 0$, $\lim \beta_i = 0$,
and furthermore, $\sum_i \gamma_i^2 = \infty$, where
$\gamma_i = \alpha_i \beta_i$.
Denote the canonical basis of $\xd(\alpha)$ by $(e_i)_{i \in \N}$.
Then $e_{2i-1} = \xi_{2i-1} \oplus 0$, and
$e_{2i} = \xi_{2i} \oplus \alpha_i \xi_{2i}$ ($(\xi_j)$ is
an orthonormal basis in $\ell_2$).
For $k \geq 0$ and $i \in [N_k, n_{k+1})$, let
\begin{equation}
f_i = \beta_i e_{2i} + \sqrt{1 - \beta_i^2} \, e_{2i-1} =
\big(\beta_i \xi_{2i} + \sqrt{1 - \beta_i^2} \, \xi_{2i-1}\big) \oplus
\gamma_i \xi_{2i} .
\label{def_f_i}
\end{equation}
We show that $Y$ is not
completely complemented in $\xd(\alpha)$. Indeed, suppose,
for the sake of contradiction, that there exists a c.b.~projection
$P$ from $\xd(\alpha)$ onto $Y$. For $\vr \in \{-1,1\}^\N$, define
an operator $\Lambda_\vr \in B(\xd(\alpha))$ by setting
$\lambda_\vr e_j = \vr_{\lceil j/2 \rceil} e_j$.
For any such $\vr$, \eqref{basis} implies that $\Lambda_\vr$ is
a complete isometry. For any $i \in \N$, we have
$\Lambda_\vr \xi = \vr_i \xi$ whenever $\xi \in \span[e_{2i-1}, e_{2i}]$.
In particular, $\Lambda_\vr f_i = \vr_i f_i$.
Let $Q = {\mathrm{Ave}}_{\vr \in \{-1,1\}^\N} \Lambda_\vr P \Lambda_\vr$.
Note that $\ran Q \subset Y$, and $Q|_Y = I_Y$, hence $Q$ is a projection
onto $Y$. Furthermore, $\|Q\|_{cb} \leq \|P\|_{cb}$.

For each $i$, we have $Q e_{2i} = a_i f_i$, and $Q e_{2i-1} = b_i f_i$.
The equations $Q f_i = f_i$ and \eqref{def_f_i} yield
$a_i \beta_i + b_i \sqrt{1 - \beta_i^2} = 1$.
Then $\sup_i \max\{|a_i|, |b_i|\} \leq \|Q\|$. As
$\lim \beta_i = 0$, there exists $K \in \N$ such that
$|b_i| > 1/2$ for $i > K$. Find $N \in \N$ s.t.
$\sum_{i = K+1}^{K + N} \gamma_i^2 > 4 \|Q\|_{cb}^2$ (this is
possible, since $\sum_i \gamma_i^2 = \infty$). Consider
$x = \sum_{i = K+1}^{K + N} E_{i1} \otimes e_{2i-1} \in M_N(\xd(\alpha))$.
Then $\|x\| = 1$, and therefore,
$$
\eqalign{
\|Q\|_{cb}
&
\geq
\|(I_{M_N} \otimes Q)x\| =
\|\sum_{i = K+1}^{K + N} E_{i1} \otimes b_i f_i\|
\cr
&
\geq
\|\sum_{i = K+1}^{K + N}
E_{i1} \otimes b_i \gamma_i \xi_{2i}\|_{M_N(\col)} =
\Big( \sum_{i = K+1}^{K + N} |b_i|^2 \gamma_i^2 \Big)^{1/2}
\cr
&
\geq
\frac{\Big( \sum_{i = K+1}^{K + N} \gamma_i^2 \Big)^{1/2}}{2}
> \|Q\|_{cb} ,
}
$$
which is impossible.
\end{remark}

\begin{remark}\label{isom_to_complem}
Suppose a Banach space $E$ is such that every infinite dimensional subspace $E$
is isomorphic to a complemented subspace of $E$. We not not know whether $E$
is necessarily isomorphic to a Hilbert space.
\end{remark}


%

\section{Completely isomorphic classification of subspaces of $X(A)$}\label{classification}


The main goal of this section is to prove Theorem~\ref{bireducible}
and Corollary~\ref{biembeddable} below.
Recall that $\comp$
is the set of compact contractions which are not Hilbert-Schmidt.

\begin{theorem}\label{bireducible}
If $A \in B(\ell_2)$ belongs to $\comp$,
then $(\bs(X(A)), \simeq)$ is Borel bireducible
to the complete $\ks$ relation.
\end{theorem}

Together with Corollary~\ref{classify_cor}, this theorem immediately implies

\begin{corollary}\label{biembeddable}
If $A \in B(\ell_2)$ belongs to $\comp$,
then the relation of complete biembeddability on
$\bs(X(A))$ is Borel bireducible to the complete $\ks$ relation.
\end{corollary}

The proof of Theorem~\ref{bireducible} proceeds in two steps.
First, we introduce the space $S_A$ of sequences of non-negative
generalized integers, with an equivalence
relation $\ssim$, and show the latter is Borel bireducible with
$(\bs(X(A)), \simeq)$. Then we prove that $(S_A, \ssim)$ is, in fact,
a complete $\ks$ relation.

Suppose $A \in B(\ell_2)$ is of class $\comp$.
By Lemma~\ref{simple_X(A)}, we can assume that $\|A\| < 1$,
and $A \geq 0$. List the positive eigenvalues of $A$ in the non-increasing
order: $1 > \|A\| = s_1^o \geq s_2^o \geq \ldots > 0$.
In the terminology of Section~\ref{sequences},
$(s_i^o)_{i \in \N} = {\mathbf{D}}(A)$.
Clearly, $\lim_i s_i^o = 0$. Let $(\xi_i)_{i \in \N}$ be
the normalized eigenvectors of $A$, corresponding to the eigenvalues
$s_i^o$. We can identify $\span[\xi_i : i \in \N]$ with $\ell_2$.
Consider the operator
$\hat{A} = \diag(s_i^o) \oplus 0 \in B(\ell_2 \oplus \ell_2)$.
By Corollary~\ref{inf_ker}, $X(A) \simeq X(\hat{A})$.
For the rest of this section, we assume that $A = \hat{A}$.

Denote by $(\xi_i^\prime)_{i \in \N}$ the canonical orthonormal
basis in the second copy of $\ell_2$. Then the canonical basis
of $X(A)$ is the collection of vectors $e_i = \xi_i \oplus s_i^o \xi_i$
and $f_i = \xi_i^\prime \oplus 0$. As in \eqref{basis_general}, we have,
for $n \times n$ matrices $a_1, b_1, a_2, b_2, \ldots$,
\begin{equation}
\eqalign{
&
\|\sum_i a_i \otimes e_i + \sum_i b_i \otimes f_i\|_{M_n(X(A))}^2
\cr
&
=
\max \big\{ \|\sum_i a_i a_i^* + \sum_i b_i b_i^*\|,
\|\sum_i s_i^{o2} a_i^* a_i\| \big\} .
}
\label{particular_basis}
\end{equation}

For an infinite dimensional $Y \hra X(A)$, we let
$(s_i(Y)) = \mathbf{D}(A|_Y)$. Recalling the definition of
$\mathbf{D}$ from Section~\ref{sequences}, we see that, if
$\rank(A|_Y) = \infty$, then $s_1(Y) \geq s_2(Y) \geq \ldots$ are the
positive singular values of $A|_Y$, listed in the non-increasing order.
In the case of $\rank(A|_Y) = n < \infty$,
$s_1(Y) \geq s_2(Y) \geq \ldots \geq s_n(Y)$
are the $n$ positive singular values of $A|_Y$, and $s_i(Y) = 0$ for $i > n$.
In this notation, $s_k^o = s_k(X(A))$.
Clearly, $s_k(Y) \leq s^o_k$ for any $k$,
and any $Y \hra X(A)$.

For $k \in \N$, set $\nn_k(Y) = \sup\{\ell \in \N : 2^{1-\ell} \geq s_k(Y)\}$.
The sequence $\nn(Y) = (\nn_k(Y))_{k \in \N}$
belongs to $\N_*^\N$, where $\N_* = \N \cup \{\infty\}$ is viewed as the
$1$-point compactification of $\N$.
For $k \in \N$ let $\alpha_k = \nn_k(X(A))$. Define $S_A$ as the set of all
elements $\beta = (\beta_i)_{i \in \N} \in \N_*^\N$, such that
(1) $\beta_k \geq \alpha_k$ for any $k$, and
(2) $\beta_1 \leq \beta_2 \leq \ldots$.
Equipping $\N_*^\N$ with its product topology, we see that
$S_A$ is closed.

For any infinite dimensional $Y \hra X(A)$,
the sequence $(\nn_k(Y))_{k \in \N}$ belongs
to $S_A$. Conversely, for any $\beta \in S_A$
there exists $Y \hra X(A)$ s.t. $\beta = \nn(Y)$.
Indeed, suppose $\beta_k \in \N_*$, $\beta_k \geq \alpha_k$ for any $k$,
and $\beta_1 \leq \beta_2 \leq \ldots$.
Let $g_i = \sin \phi_i e_i + \cos \phi_i f_i$, with
$s_i^o \sin \phi_i = 2^{-\beta_i}$. We denote $\span[g_i : i \in \N]$
by $\bfy(\beta)$, where $\beta = (\beta_i)$. By
\eqref{particular_basis}, $\nn(\bfy(\beta)) = \beta$.

Define the relation $\ssim$ on $S_A$ as follows: $\beta \ssim \gamma$
if there exists $K \in \N$ and $I \subset \N$ s.t.
$|\beta_i - \gamma_i| \leq K$ for any
$i \notin I$, and $\sum_{i \in I} (4^{-\beta_i} + 4^{-\gamma_i}) \leq K$.
By Corollary~\ref{classify_cor}, $\beta \ssim \gamma$
iff $\bfy(\beta) \simeq \bfy(\gamma)$, and conversely,
$Y \simeq Z$ iff $\nn(Y) \ssim \nn(Z)$.

\begin{proposition}\label{borel1}
$\nn$ and $\bfy$ are Borel maps.
\end{proposition}

This immediately yields:

\begin{corollary}\label{reduction1}
$(S_A, \ssim)$ and $(\bs(X(A)), \simeq)$ are Borel bireducible to each other.
\end{corollary}

\begin{proof}[Proof of Proposition~\ref{borel1}]
First we handle the map $\bfy$.
We have to show that, for any open set $U \subset X(A)$,
$\{\beta \in S_A : \bfy(\beta) \cap U \neq \emptyset\}$ is Borel. 
But $\bfy(\beta) \cap U \neq \emptyset$ iff there exist $m \in \N$ and
$\lambda_1, \ldots, \lambda_m \in \Q + i \Q$ s.t.
$\sum_{i=1}^m \lambda_i g_i \in U$. Here, the vectors $(g_i)$ come
from the definition of $\bfy$. Note that, for each $i$, $g_i$ depends
solely (and continuously) on $\beta_i$.
Therefore, for each $m$-tuple $(\lambda_i)_{i=1}^m$,
$\sum_{i=1}^m \lambda_i g_i \in U$ is an open condition on $\beta$.
Thus, $\{\beta \in S_A : \bfy(\beta) \cap U \neq \emptyset\}$ is Borel.

Now consider $\nn$. Fix $m, \beta_m \in \N$, and show that the set
of all $Y \in \bs(X(A))$, for which $\nn_m(Y) > \beta_m \in \N$, is Borel.
To this end, find a countable set $\om$ of orthonormal $m$-tuples
$\xi = (\xi_1, \ldots, \xi_m)$ in $X(A)$, with the property that,
for every $\vr > 0$, and any orthonormal $m$-tuple $(\eta_1, \ldots, \eta_m)$
in $X(A)$, there exists $\xi = (\xi_1, \ldots, \xi_m) \in \om$ s.t.
$\|\xi_i - \eta_i\| < \vr$ for any $i$. Furthermore, find a set
$\Gamma_m$ of $m$-tuples
$\gamma = (\gamma_1, \ldots, \gamma_m) \in \C^m$, dense
in the unit sphere of $\ell_2^m$.

The Minimax Principle (see e.g. \cite[p.~75]{Bh}) states that, for an
operator $T \in B(H,K)$, we have $s_m(T) > b$ iff $H$ has an
$m$-dimensional subspace $E$ such that $\|T \eta\| > b$ for any
norm one $\eta \in E$. Therefore, $\nn_m(Y) > \beta_m \in \N$
(that is, $s_m(A|_Y) \geq 2^{-\beta_m}$) iff for every $\vr > 0$
there exists an $m$-tuple of orthonormal vectors
$\eta_1, \ldots, \eta_m \in Y$, s.t.
$\|\sum_{i=1}^m \gamma_i A \eta_i\| > 2^{-\beta_m} - \vr$
whenever $\sum_{i=1}^m |\gamma_i|^2 = 1$. This, in turn,
is equivalent to the following statement:
for every $r \in \N$, there
exists $(\xi_1, \ldots, \xi_m) \in \om$ s.t., for $1 \leq i \leq m$,
$\ball(\xi_i, 1/r) \cap Y \neq \emptyset$ (here, $\ball(x,c)$
denotes the open ball of radius $c$, with the center at $x$),
and $\|\sum_i \gamma^i A \xi_i\| > 2^{-\beta_m} - 1/r$
for every $(\gamma_i)_{i=1}^m \in \Gamma_m$.
This condition is Borel, hence $\nn$ is Borel.
\end{proof}

\begin{lemma}\label{k_sigma}
$(S_A, \ssim)$ is a $\ks$ relation.
\end{lemma}

\begin{proof}
We have to show that the set
$F = \{ (\beta, \gamma) \in S_A \times S_A : \beta \ssim \gamma\}$
is a $\ks$ set, that is, a countable union of compact sets. To this end,
define a family of subsets of $S_A \times S_A$, described below.
For $K, n \in \N$ and $I_n \subset \{1, \ldots, n\}$, define
$F(K,n,I_n)$ as the set of all pairs $(\beta, \gamma)$
($\beta = (\beta_i)$, $\gamma = (\gamma_i)$) with
the property that $|\beta_i - \gamma_i| \leq K$ for
$i \in \{1, \ldots, n\} \backslash I_n$, and
$\sum_{i \in I_n} (4^{-\beta_i} + 4^{-\gamma_i}) \leq K$.
Let $F(K,n) = \cup_{I_n \subset \{1, \ldots, n\}} F(K,n,I_n)$,
and $F(K) = \cap_n F(K,n)$. It suffices to show that
$F = \cup_{K \in \N} F(K)$.
Indeed, $F(K,n,I_n)$ is a compact subset of $S_A \times S_A$,
hence so is $F(K,n)$ (as a finite union of compact sets).
Furthermore, $F(K)$ is also compact, and $\cup_K F(K)$
is $\ks$.

Show first that $F \subset \cup_K F(K)$. By definition,
$\beta \ssim \gamma$ if there exists $K \in \N$ and
$I \subset \N$ s.t. $|\beta_i - \gamma_i| \leq K$ for any $i \notin I$,
and $\sum_{i \in I} (4^{-\beta_i} + 4^{-\gamma_i}) \leq K$.
Letting $I_n = I \cap \{1, \ldots, n\}$, we see that
$(\beta, \gamma) \in F(K,n,I_n)$ for each $n$, hence
$(\beta, \gamma) \in F(K)$.

To prove the converse implication, suppose $(\beta, \gamma) \in F(K)$
for some $K$, and show that $\beta \ssim \gamma$.
Construct a tree $T \subset \{0,1\}^\N$: for each $n$,
$T \cap \{0,1\}^n$ consists
of all the sets $I_n$ s.t. $|\beta_i - \gamma_i| \leq K$ for $i \notin I_n$,
and $\sum_{i \in I_n} (4^{-\beta_i} + 4^{-\gamma_i}) \leq K$
(we identify the set of subsets of $\{1, \ldots, n\}$ with $\{0,1\}^n$).
The set $T$ is indeed a tree: if $I_n \in T$,
then $I_n \cap \{1 \ldots, m\} \in T$ for $m < n$.
By assumption, $T$ has arbitrarily long branches. By K\"onig's Lemma
\cite[p.~20]{Ke}, $T$ has an infinite branch, which yields a set
$I \subset \N$ s.t. $|\beta_i - \gamma_i| \leq K$ for $i \notin I$, and
$\sum_{i \in I} (4^{-\beta_i} + 4^{-\gamma_i}) \leq K$.
\end{proof}

Next consider a space $\Xi = \prod_{k \in \N} \Xi_k$, where
$\Xi_k = \{0, \ldots, k-1\}$, with the equivalence relation
$c \eks b$ iff $\sup_i |c_i - b_i| < \infty$ (here, $c = (c_i)_{i \in \N}$,
$b = (b_i)_{i \in \N}$). By \cite{Ro}, $\eks$ is a complete $\ks$ relation,
hence, by Lemma~\ref{k_sigma}, it reduces $(S_A, \ssim)$. It remains
to prove the converse.

\begin{proposition}\label{borel2}
There exists a Borel map
$\phi : \Xi \to S_A$ s.t. $\phi(b) \ssim \phi(c)$ iff $c E_{\ks} b$.
\end{proposition}

\begin{proof}
As $\sum_i 4^{-\alpha_i} = \infty$, and $\lim_i \alpha_i = \infty$, there exists
a sequence of positive integers $1 = p_0 < q_1 < p_2 < q_2 < \ldots$
s.t. $\sum_{i \in I_k} 4^{-\alpha_i} > 4^{2k}$ ($I_k = [p_k, q_k - 1)$),
and $\alpha_{p_{k+1}} > k + q_k$. Define $\phi((b_i)) = (b^\prime_j)$
by setting $b_j^\prime = \alpha_j + b_k$ if $j \in I_k$, and
$b_j^\prime = \min\{\alpha_j + k , \alpha_{p_{k+1}}\}$
if $q_k \leq j < p_{k+1}$. Clearly, $\phi$ is a Borel map.
Moreover, if $c \eks b$, then $\phi(b) \ssim \phi(c)$. Suppose, on the
contrary, that $b^\prime \ssim c^\prime$, where $b^\prime = \phi(b)$ and
$c^\prime = \phi(c)$. Then there exists a set $I \subset \N$ and $K \in \N$
s.t. $|b_j^\prime - c_j^\prime| \leq K$ for $j \notin I$, and
$\sum_{j \in I} (4^{-b_j^\prime} + 4^{-c_j^\prime}) \leq K$.
We shall show that $|b_k - c_k| \leq K$ for all but finitely many $k$'s.
Indeed, otherwise there exist infinitely many $k$'s s.t. $I_k \subset I$
(this follows from the fact that
$b_j^\prime - c_j^\prime = b_k - c_k$ for $j \in I_k$).
But
$$
\sum_{j \in I_k} (4^{-b_j^\prime} + 4^{-c_j^\prime}) \geq
2 \cdot 4^{-k} \sum_{j \in I_k} 4^{-\alpha_j^\prime} > 1 ,
$$
hence 
$$
\sum_{j \in I} (4^{-b_j^\prime} + 4^{-c_j^\prime}) \geq
\sum_{I_k \subset I} \sum_{j \in I_k} (4^{-b_j^\prime} + 4^{-c_j^\prime}) =
\infty ,
$$
a contradiction.
\end{proof}

\begin{proof}[Conclusion of the proof of Theorem~\ref{bireducible}]
By Corollary~\ref{reduction1}, $(\bs(X(A)), \simeq)$ and $(S_A, \ssim)$ are
Borel bireducible to each other. By Lemmas \ref{k_sigma} and \ref{borel2},
$(S_A, \ssim)$ is Borel bireducible to a complete $\ks$ relation.
\end{proof}

\begin{remark}\label{banach_classify}
For many separable Banach spaces $X$, it is known that the isomorphism
relation on $\bs(X)$ reduces certain ``classical'' relations, such as $\eks$
(see e.g. \cite{An, FG, FLRo, FRo}).
\end{remark}

\section{Proofs of Theorems \ref{intro:Ksigma},
\ref{intro:unique_basis}, \ref{intro:subbasis}}\label{main_proofs}

Recall that the class $\comp$ consists of all compact contractions,
which are not Hilbert-Schmidt, and the family $\fami$ is the set of all
operator spaces $X(A)$, where $A \in B(\ell_2)$ belongs to $\comp$.
Clearly, all these spaces are isometric to $\ell_2$.

\begin{proof}[Proof of Theorem \ref{intro:Ksigma}]
Suppose $X(A) \in \fami$. By Theorem~\ref{bireducible} and
Corollary \ref{biembeddable}, the relations of complete
isomorphism and complete biembeddability on $\bs(X(A))$
are complete $\ks$. To show that $\fami$ contains a continuum
of spaces, not completely isomorphic to each other, pick
$A \in B(\ell_2) \cap \comp$. Consider a space
$\Xi = \prod_{k \in \N} \Xi_k$, where $\Xi_k = \{0, \ldots, k-1\}$,
with the equivalence relation $c \eks b$ iff $\sup_i |c_i - b_i| < \infty$
(here, $c = (c_i)_{i \in \N}$, $b = (b_i)_{i \in \N}$). By the results
of Section~\ref{classification}, there exists a Borel map
$\Phi : \Xi \to \bs(X(A))$, such that $\Phi(b) \simeq \Phi(c)$ iff
$b \eks c$. It remains to find a family
$(b_\vr)_{\vr \in \{0,1\}^\N} \subset \Xi$, such that
$b_\vr \eks b_\delta$ iff $\vr = \delta$. To this end, write
$\N$ as a disjoint union of infinite sets $I_k$ ($k \in \N$).
For any $\vr = (\vr_k)_{k=1}^\infty$, define
$$
b_\vr(i) = \left\{ \begin{array}{ll}
   0     &   i  \in I_k, \, \vr(k) = 0   \\
   i-1   &   i  \in I_k, \, \vr(k) = 1
\end{array} \right. .
$$
Clearly, this family $(b_\vr)$ has the desired properties.
\end{proof}

\begin{proof}[Proof of Theorem \ref{intro:unique_basis}]
Consider $A \in B(\ell_2)$ of class $\comp$.
The existence of the canonical basis has been established at the
beginning of Section~\ref{sequences}, while its uniqueness follows from
Proposition~\ref{unique_basis}.
\end{proof}

\begin{proof}[Proof of Theorem \ref{intro:subbasis}]
Combine Theorem~\ref{intro:unique_basis} with Theorem~\ref{subbasis}
and Corollary~\ref{cor_subbasis}.
\end{proof}

\section{Isometric classification: proof of Theorem~\ref{intro:banach}}\label{banach}




We handle the real case.
Begin by introducing a numerical invariant of
subspaces of $X = \R \oplus_1 \ell_2$. Denote by $P$ the ``natural''
projection onto $\R$. For $Y \in \subsp(X)$, define $c(Y) = \|P|_Y\|$.

\begin{lemma}\label{attains}
For $Y \in \subsp(X)$, there exists $x \in Y$ such that
$\|x\| = 1$, and $\|P x\| = c(Y)$. Moreover, if this $x$ is
written as $x = c(Y) \oplus (1 - c(Y))\xi_0$,
then $Y = \span[x, Y^\prime]$, where
$Y^\prime = \{0 \oplus \xi : \xi \in (Y \cap \ell_2) \cap \xi_0^\perp\}$.
\end{lemma}

\begin{proof}
If $c(Y) = 0$, the statement is trivial. Suppose $c(Y) = 1$. Then, for every
$n \in \N$, there exists $t_n \in (1-1/n, 1]$ and $\xi_n \in \ell_2$ s.t.
$\|\xi_n\| = 1$, and $t_n \oplus (1-t_n) \xi_n \in X$. As $Y$ is closed,
$1 \oplus 0 \in Y$.

Next consider $c(Y) \in (0,1)$. Suppose, for the sake of contradiction,
that there is no $x$ as in the statement of the lemma. Then for every
$n \in \N$ there exist $t_n \in (c(Y) - 1/n, c(Y))$, and $\xi_n \in \ell_2$
s.t. $\|\xi_n\| = 1$, and $t_n \oplus (1-t_n) \xi_n \in Y$.
Passing to a subsequence if necessary, we can assume that $(\xi_n)$ is
a Cauchy sequence in $\ell_2$. Indeed, otherwise there exist
$n_1 < n_2 < \ldots$ and $\alpha > 0$, such that, for any $i$,
$\|\xi_{n_{i+1}} - \xi_{n_i}\| > \alpha$. By the uniform convexity of
Hilbert spaces (which follows, for instance, from the parallelogram identity),
 there exists $\beta > 0$ s.t.
$\|(1-t_{n_{i+1}}) \xi_{n_{i+1}} + (1-t_{n_i}) \xi_{n_i}\|/2 < 1 - c(Y) - \beta$
for any $i$. Define
$$
y_i = \frac{t_{n_{i+1}} + t_{n_i}}{2} \oplus
\frac{(1-t_{n_{i+1}}) \xi_{n_{i+1}} + (1-t_{n_i}) \xi_{n_i}}{2} \in Y .
$$
Then $\|y_i\| < 1 - \beta$, and $\lim_i \|P y_i\| = c(Y)$.
Therefore, $\|P|_Y\| > c(Y)$, which is impossible.

Thus, the sequence $(\xi_n)$ converges to some $\xi_0 \in \ell_2$. Then
$x = c(Y) \oplus (1-c(Y))\xi_0$ is the limit of the sequence
$t_n \oplus (1-t_n) \xi_n$, hence it belongs to $Y$.
Clearly, $\|P x\| = c(Y)$, and $Y = \span[x, Y \cap \ell_2]$.
Moreover, any $\xi \in Y \cap \ell_2$
is orthogonal to $\xi_0$. Indeed, otherwise there exists
$\xi \in Y \cap \ell_2$ and $z \in \C$ s.t. $\|\xi_0 + z \xi\| < \|\xi_0\|$.
Then $x^\prime = x + (0 \oplus \xi) = c(Y) \oplus (\xi_0 + z \xi)$
belongs to $Y$, $\|x^\prime\| < 1$, and $\|P x^\prime\| = c(Y)$,
which is impossible.
\end{proof}

For $t \in [0,1]$, define $\phi(t) = t + \sqrt{(1-t)^2 + 1}$.
Clearly, $\phi$ is continuous and increasing.

\begin{lemma}\label{weak}
For $Y \in \subsp(X)$,
$$
\phi(c(Y)) = \sup \big\{ \liminf_i \|x + y_i\| : x, y_i \in Y, \, \,
\|x\| = \|y_i\| = 1 , \, \, y_i \overset{w}{\to} 0 \big\} .
$$
Moreover, there exist a norm $1$ $x \in Y$, and a normalized
weakly null sequence $(y_i)$ in $Y$, such that
$\phi(c(Y)) = \|x + y_i\|$ for every $i$.
\end{lemma}

\begin{proof}
Assume $c(Y) \in (0,1)$ (only minimal changes are needed to
handle $c(Y) \in \{0,1\}$).
Write $x = t \oplus (1-t) \xi$ and $y_i = t_i \oplus (1-t_i) \xi_i$.
Here, $t, t_i \in [0,1]$, $\xi_i \in \ell_2$, and $\|\xi_i\| = 1$.
As $y_i \to 0$ weakly, $t_i \to 0$, and $\langle \xi, \xi_i \rangle \to 0$.
Therefore, $\lim_i \|x + y_i\| = \phi(t_i)$. Taking the supremum over all
$x \in Y$, we prove the desired equality. Furthermore, by Lemma~\ref{attains},
$Y = \span[x, Y \cap \ell_2]$, where $x = c(Y) \oplus (1-c(Y)) \xi_0$,
$\xi_0 \in \ell_2$ has norm $1$, and $Y \cap \ell_2$ is
orthogonal to $\xi_0$. Let $(\xi_i)$ be an orthonormal basis in
$Y \cap \ell_2$. Then $\phi(c(Y)) = \|x + y_i\|$ for every $i$, and
$y_i \overset{w}{\to} 0$.
\end{proof}

\begin{lemma}\label{almost_embed}
If $Y$ and $Z$ are infinite dimensional subspaces of $X$, and $Y$ is
almost isometrically embeddable into $Z$, then $c(Y) \leq c(Z)$.
\end{lemma}

\begin{proof}
By definition, for every $\lambda \in (1, 1.1)$, there exist a subspace
$W \hra Z$ and a contraction $T : Y \to W$ with $\|T^{-1}\| < \lambda$.
It suffices to show that
\begin{equation}
\phi(c(W)) \geq \lambda^{-1} \phi(c(Y)) - 2 (\lambda - 1) .
\label{compare_phi}
\end{equation}
Indeed, then we would conclude
$$
\phi(c(Z)) \geq \phi(c(W)) \geq \lambda^{-1} \phi(c(Y)) - 2 (\lambda - 1) .
$$
As the above inequality holds for any $\lambda > 1$, we
conclude that $\phi(c(Y)) \leq \phi(c(Z))$.
By the monotonicity of $\phi$, $c(Y) \leq c(Z)$.

By Lemma~\ref{weak}, there exists a normalized weakly null sequence
$(y_i)$ in $Y$, and a norm one $x \in Y$, such that
$\phi(c(Y)) = \|x + y_i\|$. In the space $W$, consider the elements
$x^\prime = Tx/\|Tx\|$, and $y_i^\prime = T y_i/\|T y_i\|$.
Then the sequence $(y_i^\prime)$ is weakly null, and
\begin{equation}
\|x^\prime + y_i^\prime\| \geq \|T(x + y_i)\| -
\|(1-\|Tx\|^{-1}) Tx\| - \|(1-\|Ty_i\|^{-1}) Ty_i\| .
\label{triangle}
\end{equation}
But $\|T^{-1}\| \|T(x + y_i)\| \geq \|x+y_i\|$, hence 
$\|T(x + y_i)\| > \lambda^{-1} \phi(c(Y))$. Furthermore,
$1 \geq \|Tx\| > \lambda^{-1}$, hence
$\|(1-\|Tx\|^{-1}) Tx\| < \lambda - 1$.
Similarly, $\|(1-\|Ty_i\|^{-1}) Ty_i\| < \lambda - 1$.
By \eqref{triangle}, $\|x^\prime + y_i^\prime\| >
\lambda^{-1} \phi(c(Y)) - 2 (\lambda - 1)$.
Applying Lemma~\ref{weak}, we obtain \eqref{compare_phi}.
\end{proof}

\begin{lemma}\label{c(Y)=c(Z)}
If $Y$ and $Z$ are infinite dimensional subspaces of $X$, and
$c(Y) = c(Z)$, then $Y$ is isometric to $Z$.
\end{lemma}

\begin{proof}
We consider the case $c(Y) = c(Z) \in (0,1)$ (the extreme
cases of $c(Y) = c(Z) \in \{0,1\}$ are handled similarly).
By Lemma~\ref{attains}, $Y$ contains a norm one
$y = c(Y) \oplus \xi_Y \in Y$ (note that $\|\xi_Y\| = 1 - c(Y)$),
s.t. $Y = \span[y, Y \cap \ell_2]$, and $Y^\prime = Y \cap \ell_2$
is orthogonal to $\xi_Y$. Thus, for any $\xi \in Y^\prime$,
$\|y + \xi\| = c(Y) + \sqrt{(1-c(Y))^2 + \|\xi\|^2}$.
Similarly, $Z = \span[z, Z^\prime]$, and
$\|z + \eta\| = c(Z) + \sqrt{(1-c(Z))^2 + \|\eta\|^2}$
for any $\eta \in Z^\prime$. As $Y^\prime$ and $Z^\prime$
are both infinite dimensional separable Hilbert spaces,
there exists an isometry $T^\prime$ from $Y^\prime$ onto $Z^\prime$.
We complete the proof by defining the isometry $T$ from $Y$ onto $Z$
by setting $T y = z$, and $T|_{Y^\prime} = T^\prime$.
\end{proof}

\begin{proof}[Proof of Theorem~\ref{intro:banach}]
By Lemmas \ref{almost_embed} and \ref{c(Y)=c(Z)}, the following statements
are equivalent for $Y, Z \in \bs(X)$: (i) $Y$ and $Z$ are isometric,
(ii) $ d(Y,Z) = 1$, (iii) $Y$ and $Z$ are isometrically bi-embeddable,
(iv) $Y$ and $Z$ are almost isometrically bi-embeddable, (v) $c(Y) = c(Z)$.
Denote that canonical basis for $\ell_2$ by $e_0, e_1, \ldots$, and
consider a map 
$$
\Phi : [0,1] \to \subsp(X) : t \to
\span[t \oplus (1-t)e_0, 0 \oplus e_1, 0 \oplus e_2, \ldots] .
$$
Then $c(\Phi(t)) = t$, hence $\Phi(t_1)$ and $\Phi(t_2)$ satisfy any
(equivalently, all) of the relations (i) -- (iv) iff $t_1 = t_2$.
It remains to prove that the maps $\Phi$ and $c$ are Borel.

To handle $\Phi$, consider an open ball $U \subset X$ with the
center at $\alpha \oplus \sum_{i=0}^N \beta_i e_i$ and radius $r$. Then
$\Phi(t) \cap U \neq \emptyset$ iff there exist
$\lambda_0, \ldots, \lambda_N \in \Q$ s.t.
$$
|\alpha - t \lambda_0| + 
\Big( |\beta_0 - (1-t) \lambda_0|^2 +
\sum_{i=1}^N |\beta_i - \lambda_i|^2 \Big)^{1/2} < r .
$$
This inequality describes a Borel subset of $[0,1]$.
As any open subset of $X$ is a countable union of open balls,
the map $\Phi$ is Borel.

To deal with $c$, consider the sets
$$
U_t = \big\{s \oplus \xi \in \R \oplus_1 \ell_2 :
|s| > t , \, \|\xi\| < \sqrt{1 - t^2} \big\}
$$
($t \in [0,1]$). Clearly, $U_t$ is an open subset of $X$, and
$c(Y) > t$ iff $Y \cap U_t \neq \emptyset$.
\end{proof}

\begin{remark}\label{other_p}
Theorem~\ref{intro:banach} holds not only for $\R \oplus_1 \ell_2$, but
also for $\R \oplus_p \ell_2$, for $1 \leq p < 2$.
\end{remark}



\end{document}